%% file: main.tex
\documentclass{siamart250211}
\input{headerSIAM}

\input{headerPerso}

\begin{document}

\maketitle

\begin{abstract}
  \noindent
  The CHDG method is a hybridizable discontinuous Galerkin (HDG) finite element method suitable for the iterative solution of time-harmonic wave propagation problems.
  Hybrid unknowns corresponding to transmission variables are introduced at the element interfaces and the physical unknowns inside the elements are eliminated, resulting in a hybridized system with favorable properties for fast iterative solution.
  In this paper, we extend the CHDG method, initially studied for the Helmholtz equation, to the time-harmonic Maxwell equations.
  We prove that the local problems stemming from hybridization are well-posed and that the fixed-point iteration naturally associated to the hybridized system is contractive.
  We propose a 3D implementation with a discrete scheme based on nodal basis functions.
  The resulting solver and different iterative strategies are studied with several numerical examples using a high-performance parallel \texttt{C++} code.
\end{abstract}

  

\section{Introduction}

The numerical solution of the time-harmonic Maxwell equations is fundamental for simulating electromagnetic wave propagation, with applications ranging from antenna design to photonics and radar systems.
In this context, finite element methods are well established as a robust and flexible solution framework, offering advantages such as high-order accuracy and geometric flexibility.
However, fine meshes may be required due to pollution effects.
For this class of problems, the finite element methods lead to large, sparse, complex-valued linear systems, which are difficult to solve with direct and iterative solution methods.
On the one hand, direct approaches are expensive and difficult to parallelize.
On the other hand, iterative strategies may require many iterations due to intrinsic properties of the problem.

Discontinuous Galerkin (DG) methods have been developed and analyzed for both time-domain
and time-harmonic electromagnetic problems, see e.g.~\cite{hesthaven2007nodal, klockner2009nodal,
pernet2021discontinuous} and \cite{PerugiaSchotzauMonk2002, chaumontfrelet2025, PerugiaSchotzau2003, Buffa2006}, respectively.
From a practical point of view, the cell local nature of the degrees of freedom (DOFs) in DG methods leads to access patterns that are favorable to high performance computing and in particular, to GPUs \cite{klockner2009nodal, hesthaven2007nodal}.
However, the number of unknowns in the resulting linear system is larger than for a corresponding continuous Galerkin counterpart, and solving these systems for time-harmonic application problems is still prohibitively expensive.

Several specialized methods have emerged to address these challenges and have been applied specifically to the time-harmonic Maxwell equations, namely hybridizable DG (HDG) methods \cite{nguyen2011hybridizable, li2013numerical, camargo2020hdg} and Trefftz approaches \cite{Huttenen2007,Hiptmair2013,fure2020discontinuous,Pernet2023}.
The HDG methods consists in introducing hybrid unknowns at the element interfaces and eliminating the physical unknowns.
In the Trefftz approach, basis functions tailored for the considered problem are used.
In both cases, the number of unknowns are reduced (sometimes drastically).
Recent studies have also focused on advanced preconditioning techniques and domain decomposition methods to improve the scalability of DG-based solvers, see, e.g., \cite{Dolean2008, Dolean2015, ElBouajaji2012, li2014hybridizable, he2016optimized} and references therein.
This is still an active area of research and the interplay between the discretization and the preconditioning strategies is very
much intertwined.

In the recent work \cite{modave2023hybridizable}, a HDG variant, termed CHDG, is studied in the
context of Helmholtz problems. The key idea is to reformulate the local DG flux in terms of the
so-called transmission variables (incoming and outgoing) and then to introduce skeleton unknowns
based on these variables. Then, an elimination process similar to the one in HDG allows to solve
only for the incoming variable. The method naturally results in a system of the form $(\opI -
\opPi\opS)\bg = \bb$ that can be solved with a fixed-point iteration, similar to the ultra weak
variational formulation (UWVF) \cite{Despres1998}. Numerical results indicate that standard
Krylov solvers such as CGNR or GMRES require fewer iterations with CHDG than with other
discretization methods.

In this paper, we extend CHDG to the time-harmonic Maxwell equations with constant coefficients, and we propose a 3D implementation based on nodal basis functions.
Appropriate transmission variables are defined in this setting, and the good properties obtained in \cite{modave2023hybridizable} are recovered.
In particular, the element-local problems associated with the hybridization are well-posed, and the fixed-point iteration is a contraction.
We give implementation details for the resulting scheme with a strategy based on nodal basis functions, which has advantages for the practical parallel implementation.
Specific preconditioned CGNR and GMRES iterations of the hybridized system are discussed, which in practice lead to faster convergence.

The rest of this paper is organized as follows.
In \S\ref{sect:DG-hybrid} we introduce the mathematical model, the standard DG method, the transmission variables, and the hybridization technique.
In \S\ref{sect:analysis} the local and hybridized problems are analyzed.
In \S\ref{sect:disc_local} we present the discrete systems with implementation details, and we discuss linear solver strategies for the hybridized problem.
In \S\ref{sect:numerics} we present numerical results on a suite of benchmarks that both validate code correctness and compare the solver strategies on reference problems.
All experiments are conducted in 3D with a parallel \texttt{C++} code.
Finally, we present conclusion and future perspectives in \S\ref{sect:conclusion}.

\section{Discontinuous Galerkin method and hybridization}
\label{sect:DG-hybrid}

Let $\Omega\subset\mathbb{R}^3$ be a domain with a polyhedral boundary partitioned into three mutually disjoint polytopal Lipschitz subsets, i.e.~$\partial\Omega = \Gamma\rmsub{E}\cup\Gamma\rmsub{H}\cup\Gamma\rmsub{I}$.
In this paper, we consider the time-harmonic Maxwell equations completed with standard boundary conditions
\begin{align}
  \left\{
  \begin{aligned}
    \im\kappa\be  - \curl\bh
     & = \bzero       &  & \text{in }\Omega,          \\
    \im\kappa\bh  + \curl\be
     & = \bzero       &  & \text{in }\Omega,          \\
    \bn\times\be
     & = \bs\rmsub{E} &  & \text{on }\Gamma\rmsub{E}, \\
    \bn\times\bh
     & = \bs\rmsub{H} &  & \text{on }\Gamma\rmsub{H}, \\
    -\bn\times(\bn\times\be) +  \bn\times\bh
     & = \bs\rmsub{I} &  & \text{on }\Gamma\rmsub{I},
  \end{aligned}
  \right.
  \label{eq:maxwell_system}
\end{align}
where the unknowns $\be:\Omega\to\mathbb{C}^3$ and $\bh:\Omega\to\mathbb{C}^3$
represent the electric and magnetic field respectively, $\im$ is the imaginary unit,
$\kappa > 0$ represents the (constant) wave number, $\bn$ is the outward
unit normal vector with respect to $\partial\Omega$, and
$\bs\rmsub{E}$, $\bs\rmsub{E}$ and $\bs\rmsub{E}$ are given surface sources. Volumetric sources are
easily handled as well.

\subsection{Discretization and notation}

We consider a conforming mesh $\CT_h$ of the domain $\Omega$ consisting of simplicial
elements $K$.
The collection of element boundaries is denoted by
$\partial\CT_h := \{\partial K \:|\: K\in\CT_h\}$,
the collection of faces of the mesh is denoted by $\CF_h$,
and the collection of faces of an element $K$ is denoted by $\CF_K$.

The numerical physical fields, denoted by $\be_h$ and $\bh_h$, belong to the broken space
\begin{displaymath}
  \BV_h := \prod_{K\in\CT_h} \BCP_p(K),
\end{displaymath}
where $p \geq 0$ is a given polynomial degree and $\BCP_p(\cdot)$ is the space
of complex vector-valued polynomials of degree smaller than or equal to $p$.
For a face $F$ of an element $K$, we shall use the \emph{tangential trace}
and \emph{tangential component trace}, which, for a vector-valued polynomial $\bv_K\in\BCP_p(K)$, are defined as
\begin{align}
  \bn_{K,F}\times\bv_K
  \quad\text{and}\quad
  \bpi_{K,F}^t(\bv_K) := -\bn_{K,F}\times(\bn_{K,F}\times\bv_K),
  \label{eq:tan_trace_defn}
\end{align}
where $\bn_{K,F}$ denotes the outward unit normal to $F$.
These traces verify the identity $\bn_{K,F}\times\bpi_{K,F}^t(\bv_K) = \ntimesKF{\bv_K}$,
and they belong to the \emph{tangential trace space}
\begin{align}
  \BCP_p^t(F) := \{\bv\in\BCP_p(F) \:|\: \bn_F\cdot\bv = 0\}.
  \label{eq:tantrace_defn}
\end{align}
Let us note that the traces \eqref{eq:tan_trace_defn} can be applied
to any tangential vector-valued polynomial $\bv^t_F\in\BCP_p^t(F)$ without ambiguities.
The orientation of the normal $\bn_F$ depends on the context.

We introduce the following sesquilinear forms for complex vector-valued functions
\begin{displaymath}
  \begin{aligned}
    \prodV{\bv}{\bw}_{K}             & := \int_K \bv\cdot\overline{\bw}\:d\bx,           &
    \prodV{\bv}{\bw}_{\CT_h}         & := \sum_{K\in\CT_h} \prodV{\bv}{\bw}_{K},
    \\
    \prodS{\bv}{\bw}_{F}             & := \int_F \bv\cdot\overline{\bw}\:d\sigma(\bx),   &
    \prodS{\bv}{\bw}_{\CF_h}         & := \sum_{F\in\CF_h} \prodS{\bv}{\bw}_{F},
    \\
    \prodS{\bv}{\bw}_{\partial K}    & := \sum_{F\in\CF_K} \prodS{\bv}{\bw}_{F},         &
    \prodS{\bv}{\bw}_{\partial\CT_h} & := \sum_{K\in\CT_h} \prodS{\bv}{\bw}_{\partial K}.
  \end{aligned}
\end{displaymath}
By convention, the quantities used in the surface integral $\prodS{\cdot}{\cdot}_{\partial K}$
are associated to the faces of $K$, e.g.~$\bn_{K,F}$ and $\bpi_{K,F}^t(\bv_K)$ with $F\in\CF_K$.
The subscripts $K$ and $F$ are dropped when there is no danger of ambiguity from context.
For $\bv,\bw\in\BCP_p(K)$, we have the integration by part formula
\begin{displaymath}
  \prodV{\curl\bw}{\bv}_K - \prodV{\bw}{\curl\bv}_K =
  \prodS{\bn\times\bw}{\bpi^t(\bv)}_{\partial K}
\end{displaymath}
and the following identities
\begin{displaymath}
  \begin{aligned}
    \prodS{\bpi_{K,F}^t(\bv)}{\ntimesKF{\bw}}_{F} & =
    -\prodS{\ntimesKF{\bv}}{\bpi_{K,F}^t(\bw)}_{F},
    \\
    \prodS{\ntimesKF{\bv}}{\ntimesKF{\bw}}_{F}    & =
    \prodS{\bpi_{K,F}^t(\bv)}{\bpi_{K,F}^t(\bw)}_{F}.
  \end{aligned}
\end{displaymath}

\subsection{Standard DG formulation with upwind fluxes}

The standard DG formulation for the time-harmonic Maxwell equations is given by
\begin{problem}
\label{pbm:standardDG}
Find $(\be_h,\bh_h) \in \BV_h\times\BV_h$ such that, for all $(\bv_h,\bw_h) \in
  \BV_h\times\BV_h$,
\begin{displaymath}
  \left\{
  \begin{aligned}
    \im\kappa\prodV{\be_h}{\bv_h}_{\CT_h}
    - \prodV{\bh_h}{\curl \bv_h}_{\CT_h}
    - \prodS{\bn\times\widehat{\bh}^t(\be_h,\bh_h)}{\bpi^t(\bv_h)}_{\partial\CT_h}
     & = 0, \\
    \im\kappa\prodV{\bh_h}{\bw_h}_{\CT_h}
    + \prodV{\be_h}{\curl \bw_h}_{\CT_h}
    + \prodS{\bn\times\widehat{\be}^t(\be_h,\bh_h)}{\bpi^t(\bw_h)}_{\partial\CT_h}
     & = 0,
  \end{aligned}
  \right.
\end{displaymath}
with the \emph{numerical fluxes} $\bn\times\widehat{\bh}^t(\be_h,\bh_h)$ and
$\bn\times\widehat{\be}^t(\be_h,\bh_h)$ defined below.
\end{problem}
We choose \textit{upwind fluxes}, see e.g.~\cite{hesthaven2007nodal, pernet2021discontinuous}.
For each interior face $F\not\subset\partial\Omega$ of each element $K$, they are given by
\begin{subequations}
  \begin{align}
    \left\{
    \begin{aligned}
      \ntimesKF{\widehat{\be}_F^t}
       & := \ntimesKF{(\be_K + \be_{K'})/2}
      +  \bpi_{K,F}^t (\bh_K - \bh_{K'})/2,
      \\
      \ntimesKF{\widehat{\bh}_F^t}
       & := \ntimesKF{(\bh_K + \bh_{K'})/2}
      -  \bpi_{K,F}^t (\be_K - \be_{K'})/2,
    \end{aligned}
    \right.
    \label{eq:upwind_flux_int}
  \end{align}
  where $K'$ is the neighboring element of $K$ sharing the face $F$.
  For a boundary face $F$, they are defined as
  \begin{align}
     &
    \left\{
    \begin{aligned}
      \bn_{K,F}\times\widehat{\be}_F^t
       & := \bs\rmsub{E}
      \\
      \bn_{K,F}\times\widehat{\bh}_F^t
       & := \ntimesKF{\bh_K} - \left(\bpi_{K,F}^t(\be_K) + \bn_{K,F}\times\bs\rmsub{E}\right)
    \end{aligned}
    \right.
     &
     & \text{if }F\subset\Gamma\rmsub{E},
    \label{eq:dg_fluxes_E}
    \\
     &
    \left\{
    \begin{aligned}
      \bn_{K,F}\times\widehat{\be}_F^t
       & := \bn_{K,F}\times\be_K +  \left(\bpi_{K,F}^t(\bh_K) + \bn_{K,F}\times\bs\rmsub{H}\right)
      \\
      \bn_{K,F}\times\widehat{\bh}_F^t
       & := \bs\rmsub{H}
    \end{aligned}
    \right.
     &
     & \text{if }F\subset\Gamma\rmsub{H},
    \label{eq:dg_fluxes_H}
    \\
     &
    \left\{
    \begin{aligned}
      \ntimesKF{\widehat{\be}_F^t}
       & := \ntimesKF{\left(\bpi_{K,F}^t(\be_K) -  \bn_{K,F}\times\bh_K + \bs\rmsub{I}\right)/2}
      \\
      \bn_{K,F}\times\widehat{\bh}_F^t
       & := - \left(\bpi_{K,F}^t(\be_K) -  \bn_{K,F}\times\bh_K - \bs\rmsub{I}\right)/2
    \end{aligned}
    \right.
     &
     & \text{if }F\subset\Gamma\rmsub{I}.
    \label{eq:dg_fluxes_I}
  \end{align}
\end{subequations}
These numerical fluxes satisfy
\begin{displaymath}
  \begin{aligned}
    \bpi_{K,F}^t(\widehat{\be}_F^t) - \bn_{K,F}\times\widehat{\bh}_F^t
     & = \bpi_{K,F}^t(\be_K) - \ntimesKF{\bh_K},      \\
    \bn_{K,F}\times\widehat{\be}_F^t + \bpi_{K,F}^t(\widehat{\bh}_F^t)
     & = \bn_{K,F}\times\be_K + \bpi_{K,F}^t(\bh_K),
  \end{aligned}
\end{displaymath}
regardless of whether $F$ is an interior or boundary face.

\subsection{Transmission variables}

The definition of the upwind fluxes is related to the characteristic analysis of the system, which exhibits quantities transported through the faces of the elements.

We consider the time-dependent Maxwell equations
\begin{displaymath}
  \left\{
  \begin{aligned}
    \frac{\partial \be}{\partial t} - \curl{\bh} & = \bzero, \\
    \frac{\partial \bh}{\partial t} + \curl{\be} & = \bzero.
  \end{aligned}
  \right.
\end{displaymath}
Assuming only the variation with respect to a constant direction $\bn$ leads to
system
\begin{displaymath}
  \left\{
  \begin{aligned}
    \frac{\partial \be}{\partial t}
    - \bn\times\frac{\partial\bh}{\partial n}
     & = \bzero,
    \\
    \frac{\partial \bh}{\partial t}
    + \bn\times\frac{\partial\be}{\partial n}
     & = \bzero.
  \end{aligned}
  \right.
\end{displaymath}
Applying $-\bn\times(\bn\times\cdot)$ and $-(\bn\times\cdot)$ to the first and second equations, respectively, and combining the resulting equations yields the following decoupled transport equations
\begin{displaymath}
  \frac{\partial \bq^\pm}{\partial t} \pm \frac{\partial\bq^\pm}{\partial n} = \bzero,
\end{displaymath}
where the transported quantities are
$\bq^{\pm} := -\bn\times(\bn\times\be)\mp \bn\times\bh$,
which are frequently called the \emph{characteristic variables}.

\begin{figure}[!tb]
  \centering
  \includegraphics{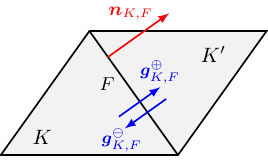}
  \caption{Notation for the outgoing and incoming transmission variables at the face $F$ shared by an element $K$ and a neighboring element $K'$.
    }
  \label{fig:notation}
\end{figure}

The upwind fluxes can be rewritten by introducing transmission variables related to the quantities transported through the faces of the elements.
At each face $F$ of an element $K$, we define the \emph{outgoing transmission variable} as
\begin{subequations}
  \begin{align}
    \bg^{\oplus}_{K,F}
     & := \bpi_{K,F}^t(\be_K) - \ntimesKF{\bh_K},
    \label{eq:transvar_defs_p}
  \end{align}
  and the \emph{incoming transmission variable} as
  \begin{align}
    \bg_{K,F}^\ominus
    :=
    \begin{cases}
      \bg^{\oplus}_{K',F},
       & \text{if }F\not\subset\partial\Omega\text{ is shared by } K\text{ and }K',
      \\
      - \bg_{K,F}^\oplus -2\bn_{K,F}\times\bs\rmsub{E},
       & \text{if }F\subset\Gamma\rmsub{E},
      \\
      \bg_{K,F}^\oplus + 2\bs\rmsub{H},
       & \text{if }F\subset\Gamma\rmsub{H},
      \\
      \bs\rmsub{I},
       & \text{if }F\subset\Gamma\rmsub{I}.
    \end{cases}
    \label{eq:transvar_defs_m}
  \end{align}
\end{subequations}
This notation is illustrated in Figure~\ref{fig:notation}.
By using the transmission variables \eqref{eq:transvar_defs_p}-\eqref{eq:transvar_defs_m}, the numerical fluxes \eqref{eq:upwind_flux_int}-\eqref{eq:dg_fluxes_I} can be rewritten as
\begin{align}
  \left\{
  \begin{aligned}
    \bn_{K,F}\times \widehat{\be}_F^t
     & = \bn_{K,F}\times(\bg_{K,F}^{\oplus} + \bg_{K,F}^{\ominus})/2,
    \\
    \bn_{K,F}\times\widehat{\bh}_F^t
     & = -(\bg_{K,F}^{\oplus} - \bg_{K,F}^{\ominus})/2,
  \end{aligned}
  \right.
  \label{eq:dg_flux_trans}
\end{align}
for any face $F\in\mathcal{F}_K$.

For an interior face ($F\not\subset\partial\Omega$), the outgoing variable of one side corresponds to the incoming one of the other side, i.e.~$\bg_{K,F}^{\oplus}=\bg_{K',F}^{\ominus}$ and $\bg_{K,F}^{\ominus}=\bg_{K',F}^{\oplus}$.
For a boundary face, the boundary condition is prescribed by a specific definition of the incoming variable.
For example, for a condition on the tangential trace of a field ($F\subset\Gamma\rmsub{E}$ or $\Gamma\rmsub{H}$), the incoming variable is such that the numerical flux corresponding to the field is equal to the surface source.

\subsection{Hybridization}

In the standard hybridization approach, a new variable corresponding to the numerical trace
associated to one of the field is introduced, see e.g.~\cite{li2013numerical, li2014hybridizable}
and references therein.
This variable, denoted by $\boldsymbol{\zeta}_h$, belongs to the space of the tangential vector-valued polynomials defined on the faces of the mesh,
\begin{displaymath}
  \widehat{\BV}_h^t := \prod_{F\in\CF_h}\BCP_p^t(F).
\end{displaymath}
It can be defined such that $\bn_{K,F}\times\boldsymbol{\zeta}_{F} := \bn_{K,F}\times\widehat{\bh}_F^t$ for each face $F$ of each element $K$.
The unknowns associated with the physical fields can then be eliminated from the system, resulting in a hybridized system defined on the skeleton.

As in \cite{modave2023hybridizable}, we propose an alternative approach where the new variable, denoted by $\bg^\ominus_h$, corresponds to the incoming transmission variable.
It belongs to the space
\begin{align}
  \BG_h^t := \prod_{K\in\CT_h}\prod_{F\in\CF_K}\BCP_p^t(F).
  \label{eq:Ghat_defn}
\end{align}
By contrast with the standard approach, there are two tangential vector values for each interior face instead of one, i.e.~$\bg^\ominus_{K,F}$ and $\bg^\ominus_{K',F}$ for a face $F$ shared by neighboring elements $K$ and $K'$.
There is still one value for each boundary face.

The CHDG formulation is obtained by introducing the new variable $\bg^\ominus_h$ in Problem~\ref{pbm:standardDG}, and by using \eqref{eq:dg_flux_trans}.
To write the formulation, we define the \textit{global exchange operator} $\opPi:\BG_h^t \to\BG_h^t$ given, for each face $F$ of each element $K$, by
\begin{align}
  \opPi(\bg^\oplus)|_{K,F} :=
  \begin{cases}
    \bg_{K',F}^{\oplus} & \text{if } F\not\subset\partial\Omega\text{ is shared by } K\text{ and }K',
    \\
    -\bg_{K,F}^{\oplus} & \text{if } F\subset \Gamma\rmsub{E},
    \\
    \bg_{K,F}^{\oplus}  & \text{if } F\subset \Gamma\rmsub{H},
    \\
    \bzero              & \text{if } F \subset\Gamma\rmsub{I},
  \end{cases}
  \label{eq:exchange_operator}
\end{align}
for any $\bg^\oplus\in\BG_h^t$. We also define the \textit{global right-hand side} $\bb$
given, for each face $F$ of each element $K$, by
\begin{displaymath}
  \bb|_{K,F} :=
  \begin{cases}
    \bzero                        & \text{if } F\not\subset\partial\Omega, \\
    -2\bn_{K,F}\times\bs\rmsub{E} & \text{if } F\subset\Gamma\rmsub{E},    \\
    2\bs\rmsub{H}                 & \text{if } F\subset\Gamma\rmsub{H},    \\
    \bs\rmsub{I}                  & \text{if } F\subset\Gamma\rmsub{I}.    \\
  \end{cases}
\end{displaymath}
The CHDG formulation then reads as follows.
\begin{problem}[CHDG formulation]
\label{pbm:CHDG}
Find $(\be_h,\bh_h, \bg_h^{\ominus}) \in \BV_h\times\BV_h\times\BG_h^t$
such that, for all $(\bv_h,\bw_h, \bxi_h) \in
  \BV_h\times\BV_h\times\BG_h^t$,
\begin{displaymath}
  \left\{
  \begin{aligned}
    \im\kappa\prodV{\be_h}{\bv_h}_{\CT_h}
    - \prodV{\bh_h}{\curl \bv_h}_{\CT_h}
    - \prodS{-\tfrac{1}{2}(\bg^{\oplus}(\be_h,\bh_h) - \bg_h^{\ominus})}{\bpi^t\bv_h}_{\partial\CT_h}
     & = 0, \\
    \im\kappa\prodV{\bh_h}{\bw_h}_{\CT_h}
    + \prodV{\be_h}{\curl \bw_h}_{\CT_h}
    + \prodS{\tfrac{1}{2}\bn\times(\bg^{\oplus}(\be_h,\bh_h) + \bg_h^{\ominus})}{\bpi^t\bw_h}_{\partial\CT_h}
     & = 0,
  \end{aligned}
  \right.
\end{displaymath}
and
\begin{align}
  \prodS{\bg_h^{\ominus} - \opPi(\bg^{\oplus}(\be_h,\bh_h))}{\bxi_h}_{\partial\CT_h} =
  \prodS{\bb}{\bxi_h}_{\partial\CT_h},
  \label{eq:chdg_bdry}
\end{align}
with $\bg^\oplus(\be_h,\bh_h) := \bpi^t(\be_h) - \ntimes{\bh_h}$.
\end{problem}

After eliminating the physical variables, we are left with solving \eqref{eq:chdg_bdry} on the skeleton with only the variable $\bg^\ominus_h$ as unknown.
The elimination implicitly requires the solution of local element-wise problems, where the incoming transmission variable is considered data.
The local problem for an element $K$ is given by
\begin{problem}[Local problem]
\label{pbm:CHDG_local}
Find
$(\be_K,\bh_K) \in \BCP_p(K)\times\BCP_p(K)$
such that, for all
$(\bv_K,\bw_K) \in \BCP_p(K)\times\BCP_p(K)$,
\begin{align}
  \left\{
  \begin{aligned}
     & \im\kappa\prodV{\be_K}{\bv_K}_K
    - \prodV{\bh_K}{\curl \bv_K}_K
    + \sum_{F\in\CF_K}\prodS{\frac{1}{2}\bg^{\oplus}_{K,F}(\be_K, \bh_K)}{\bpi_{K,F}^t(\bv_K)}_{F}
    \\
     & \qquad
    = \sum_{F\in\CF_K}\prodS{\frac{1}{2}\bg^{\ominus}_{K,F}}{\bpi_{K,F}^t(\bv_K)}_{F},
    \\
     & \im\kappa\prodV{\bh_h}{\bw_K}_{K}
    + \prodV{\be_K}{\curl \bw_K}_{K}
    + \sum_{F\in\CF_K}\prodS{\frac{1}{2}\bn_{K,F}\times\bg^{\oplus}_{K,F}(\be_K, \bh_K)}{\bpi_{K,F}^t(\bw_K)}_{F}
    \\
     & \qquad
    = - \sum_{F\in\CF_K}\prodS{\frac{1}{2}\bn_{K,F}\times\bg^{\ominus}_{K,F}}{\bpi_{K,F}^t(\bw_K)}_{F},
  \end{aligned}
  \right.
  \label{eq:chdg:localSys}
\end{align}
with $\bg^\oplus_{K,F}(\be_K,\bh_K) := \bpi_{K,F}^t(\be_K) - \ntimesKF{\bh_K}$, for given data $\bg^\ominus_{K,F}\in\BCP_p^t(F)$, for all $F\in\CF_K$.
\end{problem}

The global problem resulting from the elimination of the physical variables is called the \textit{hybridized problem}.
This problem can be written in a convenient abstract form by introducing the \textit{global scattering operator} $\opS:\BG_h^t\to\BG_h^t$ satisfying, for each face $F$ of each element $K$,
\begin{align}
  \label{eq:scatter_operator}
  \opS(\bg^\ominus_h)|_{K,F} := \bpi^t_{K,F}(\be_K(\bg^\ominus_h)) -
  \bn_{K,F}\times\bh_K(\bg^\ominus_h),
\end{align}
where $(\be_K,\bh_K)$ is a solution of the local problem \ref{pbm:CHDG_local} with $(\bg^\ominus_h|_{K,F})_{F\in\CF_K}$ as surface data.
The application of $\opS$ involves the solution of all the local problems.
The hybridized problem can then be written as
\begin{problem}[Hybridized problem \texttt{I}]
\label{pbm:trans_weak}
Find $\bg^\ominus_h\in\BG_h^t$ such that, for all $\bxi_h\in\BG_h^t$,
\begin{displaymath}
  \prodS{\bg_h^\ominus - \opPi(\opS(\bg_h^\ominus))}{\bxi_h}_{\CT_h} =
  \prodS{\bb}{\bxi_h}_{\CT_h}.
\end{displaymath}
\end{problem}
The problem can be written in operator form by introducing the \textit{global
  projected right-hand side} $\bb_h$ is the $\BL^2$-orthogonal projection of $\bb$ satisfying $\prodS{\bb_h}{\bxi_h}_{\partial\CT_h} =
  \prodS{\bb}{\bxi_h}_{\partial\CT_h}$ for all $\bxi_h\in\BG_h^t$.
We then obtain
\begin{problem}[Hybridized problem \texttt{II}]
\label{pbm:trans_operator}
Find $\bg_h^\ominus\in\BG_h^t$ such that
\begin{align}
  (\opI - \opPi\opS)\bg_h^\ominus = \bb_h.
\label{eqn:trans_operator}
\end{align}
\end{problem}
Problems~\ref{pbm:trans_weak} and \ref{pbm:trans_operator} are equivalent to
Problem~\ref{pbm:CHDG} because the element-wise problems are well-posed, see below.

\section{Analysis of the CHDG problems}
\label{sect:analysis}

In this section, we extend the results obtained in \cite{modave2023hybridizable} for the Helmholtz equation to the time-harmonic Maxwell equations.

\subsection{Local problem}

We show that the local problems (Problem~\ref{pbm:CHDG_local}) are well-posed, which authorizes the proposed hybridization.
\begin{theorem}
  \label{thm:wellposed}
  Problem~\ref{pbm:CHDG_local} is well-posed.
\end{theorem}
\begin{proof}
  It is sufficient to show that
  $\bg^\ominus_{K,F}=\bzero$ implies $\be_K=\bh_K=\bzero$.
  For brevity, we drop the subscripts $K$ and $F$.
  Testing system \eqref{eq:chdg:localSys} with $\bv=\be$ and $\bw=\bh$,
  and avoiding the source terms yields
  \begin{displaymath}
    \begin{aligned}
      \im\kappa\prodV{\be}{\be}_K
      - \prodV{\bh}{\curl\be}_K
      + \prodS{\tfrac{1}{2}(\bpi^t\be - \bn\times\bh)}{\bpi^t\be}_{\partial K}
       & = 0, \\
      \im\kappa\prodV{\bh}{\bh}_{K}
      + \prodV{\be}{\curl\bh}_{K}
      + \prodS{\tfrac{1}{2}\bn\times(\bpi^t\be - \bn\times\bh)}{\bpi^t\bh}_{\partial K}
       & = 0,
    \end{aligned}
  \end{displaymath}
  which can be rewritten as
  \begin{displaymath}
    \begin{aligned}
      \im\kappa\prodV{\be}{\be}_K
      - \prodV{\bh}{\curl\be}_K
      - \tfrac{1}{2} \prodS{\bn\times\bh}{\bpi^t\be}_{\partial K}
      + \tfrac{1}{2} \prodS{\bpi^t\be}{\bpi^t\be}_{\partial K}
       & = 0, \\
      \im\kappa\prodV{\bh}{\bh}_{K}
      + \prodV{\be}{\curl\bh}_{K}
      + \tfrac{1}{2} \prodS{\bn\times\be}{\bpi^t\bh}_{\partial K}
      + \tfrac{1}{2} \prodS{\bpi^t\bh}{\bpi^t\bh}_{\partial K}
       & = 0,
    \end{aligned}
  \end{displaymath}
  Taking the conjugate of both equations, integrating by part and using surface identities gives
  \begin{displaymath}
    \begin{aligned}
      - \im\kappa\prodV{\be}{\be}_K
      - \prodV{\be}{\curl\bh}_K
      - \tfrac{1}{2} \prodS{\bn\times\be}{\bpi^t\bh}_{\partial K}
      + \tfrac{1}{2} \prodS{\bpi^t\be}{\bpi^t\be}_{\partial K}
       & = 0, \\
      - \im\kappa\prodV{\bh}{\bh}_{K}
      + \prodV{\bh}{\curl\be}_{K}
      + \tfrac{1}{2} \prodS{\bn\times\bh}{\bpi^t\be}_{\partial K}
      + \tfrac{1}{2} \prodS{\bpi^t\bh}{\bpi^t\bh}_{\partial K}
       & = 0,
    \end{aligned}
  \end{displaymath}
  Adding the four previous equations yields
  \begin{displaymath}
    \|\bpi^t\be\|_{\partial K}^2 +  \|\bpi^t\bh\|_{\partial K}^2 = 0,
  \end{displaymath}
  which gives $\bpi^t\be = \bpi^t\bh = \bzero$ on $\partial K$.
  Therefore, Problem~\ref{pbm:CHDG_local} combined with these boundary conditions implies
  the pair $(\be,\bh)$ is a solution to the strong Maxwell system \eqref{eq:maxwell_system} in $K$ with
  homogeneous boundary conditions. Then, the stability of the Maxwell system implies that the pair
  $(\be,\bh)$ must also be identically zero in $K$, as desired.
\end{proof}

\subsection{Hybridized problem}
\label{sect:analysis:global}

We show that the operator $\opPi\opS$ is a strict contraction.
As a consequence of the Banach fixed-point theorem, the  hybridized problem (Problem~\ref{pbm:trans_operator}) is well-posed, and it can be solved with the fixed-point iteration.
We begin with the following lemma pertaining to the scattering operator $\opS$.
\begin{lemma}
  \label{lem:scatter_contract}
  (i) The solution of the local problem~\ref{pbm:CHDG_local} verifies
  \begin{multline}
    \sum_{F\in\CF_K}\|\bpi^t_{K,F}(\be_K) - \bn_{K,F}\times\bh_K\|_F^2
    \\
    + \sum_{F\in\CF_K}\|\bg_{K,F}^\ominus - (\bpi^t_{K,F}(\be_K) + \bn_{K,F}\times\bh_K) \|_F^2
    = \sum_{F\in\CF_K}\|\bg_{K,F}^\ominus\|_F^2.
    \label{eq:contract_equality}
  \end{multline}
  (ii) Moreover, the second term of the left-hand side of \eqref{eq:contract_equality} vanishes if and
  only if $\bg^\ominus_{K,F}=\bzero$ for all $F\in\CF_K$.
\end{lemma}
\begin{proof}
  We suppress the subscripts $K,F$ for simplicity.

  \emph{(i)} Taking the test functions
  $\bv = \be$ and $\bw = \bh$, Problem~\ref{pbm:CHDG_local} reduces to
  \begin{displaymath}
    \begin{aligned}
      \im\kappa\prodV{\be}{\be}_K
      - \prodV{\bh}{\curl\be}_K
      - \tfrac{1}{2} \prodS{\bn\times\bh}{\bpi^t\be}_{\partial K}
      + \tfrac{1}{2} \prodS{\bpi^t\be}{\bpi^t\be}_{\partial K}
       & = \tfrac{1}{2} \prodS{\bg^{\ominus}}{\bpi^t\be}_{\partial K}, \!\!\!
      \\
      \!\!\!
      \im\kappa\prodV{\bh}{\bh}_{K}
      + \prodV{\be}{\curl\bh}_{K}
      + \tfrac{1}{2} \prodS{\bn\times\be}{\bpi^t\bh}_{\partial K}
      + \tfrac{1}{2} \prodS{\bpi^t\bh}{\bpi^t\bh}_{\partial K}
       & = \tfrac{1}{2} \prodS{\bg^{\ominus}}{\bn\times\bh}_{\partial K}.
    \end{aligned}
  \end{displaymath}
  Taking the conjugate of both equations, integrating by part and using surface identities gives
  \begin{displaymath}
    \begin{aligned}
      - \im\kappa\prodV{\be}{\be}_K
      - \prodV{\be}{\curl\bh}_K
      - \tfrac{1}{2} \prodS{\bn\times\be}{\bpi^t\bh}_{\partial K}
      + \tfrac{1}{2} \prodS{\bpi^t\be}{\bpi^t\be}_{\partial K}
       & = \tfrac{1}{2} \prodS{\bpi^t\be}{\bg^{\ominus}}_{\partial K}, \!\!\!
      \\
      \!\!\!
      - \im\kappa\prodV{\bh}{\bh}_{K}
      + \prodV{\bh}{\curl\be}_{K}
      + \tfrac{1}{2} \prodS{\bn\times\bh}{\bpi^t\be}_{\partial K}
      + \tfrac{1}{2} \prodS{\bpi^t\bh}{\bpi^t\bh}_{\partial K}
       & = \tfrac{1}{2} \prodS{\bn\times\bh}{\bg^{\ominus}}_{\partial K}.
    \end{aligned}
  \end{displaymath}
  Adding the four previous equations yields
  \begin{displaymath}
    \prodS{\bpi^t\be}{\bpi^t\be}_{\partial K}
    +  \prodS{\bpi^t\bh}{\bpi^t\bh}_{\partial K}
    = \tfrac{1}{2} \prodS{\bg^{\ominus}}{ \bpi^t\be + \bn\times\bh}_{\partial K}
    + \tfrac{1}{2} \prodS{ \bpi^t\be + \bn\times\bh}{\bg^{\ominus}}_{\partial K},
  \end{displaymath}
  and then
  \begin{displaymath}
    2 \|\bpi^t\be\|^2_{\partial K} + 2 \| \bn\times\bh\|^2_{\partial K}
    = \prodS{\bg^{\ominus}}{\bpi^t\be + \bn\times\bh}_{\partial K}
    +\prodS{\bpi^t\be + \bn\times\bh}{\bg^{\ominus}}_{\partial K}.
  \end{displaymath}
  The left-hand side of this equation is equal to
  \begin{displaymath}
    \|\bpi^t\be +  \bn\times\bh\|^2_{\partial K}
    + \|\bpi^t\be -  \bn\times\bh\|^2_{\partial K},
  \end{displaymath}
  and the right-hand side is equal to
  \begin{displaymath}
    \|\bg^{\ominus}\|^2_{\partial K}
    + \|\bpi^t\be + \bn\times\bh\|^2_{\partial K}
    - \|\bg^{\ominus} - (\bpi^t\be + \bn\times\bh)\|^2_{\partial K}.
  \end{displaymath}
  We then have
  \begin{displaymath}
    \|\bpi^t\be - \bn\times\bh\|^2_{\partial K}
    = \|\bg^{\ominus}\|^2_{\partial K}
    - \|\bg^{\ominus} - (\bpi^t\be + \bn\times\bh)\|^2_{\partial K},
  \end{displaymath}
  which yields \eqref{eq:contract_equality}.

  \emph{(ii)} If the second term of \eqref{eq:contract_equality} is zero, we have that
  \begin{displaymath}
    \bg^\ominus = \bpi^t\be + \ntimes{\bh}
    \quad\text{on }\partial K.
  \end{displaymath}
  Applying this in
  Problem~\ref{pbm:CHDG_local} yields that the pair $(\be, \bh)$ solves the following system
  \begin{displaymath}
    \begin{aligned}
      \im\kappa\prodV{\be}{\bv}_K
      - \prodV{\bh}{\curl \bv}_K
      -\prodS{\ntimes{\bh}}{\bpi^t(\bv)}_{\partial K}
       & = 0,
      \\
      \im\kappa\prodV{\bh}{\bw}_{K}
      + \prodV{\be}{\curl \bw}_{K}
      + \prodS{\ntimes{\be}}{\bpi^t(\bw)}_{\partial K}
       & = 0,
    \end{aligned}
  \end{displaymath}
  for all $\bv,\bw\in\BCP_p(K)$.
  By using integration by parts in both equations, we see that the pair $(\be, \bh)$
  satisfies the problem in the strong form. However, there cannot be a nontrivial
  polynomial solution to the time-harmonic Maxwell equations.
  Hence, $\be = \bh = \bzero$, and then $\bg^\ominus = \bzero$.
  The converse follows from the well-posedness of the problem (Theorem~\ref{thm:wellposed}).
\end{proof}
We are now in a position to prove the following.
\begin{theorem}
  The scattering operator $\opS$ is a strict contraction, i.e.
  \begin{displaymath}
    \|\opS(\bg_h^\ominus)\|_{\partial\CT_h} < \|\bg_h^\ominus\|_{\partial\CT_h},\quad\forall\bg_h^\ominus\in\BG_h^t\setminus\{\bzero\}.
  \end{displaymath}
\end{theorem}
\begin{proof}
  Let $\bg_h^\ominus\in\BG_h^t\setminus\{\bzero\}$ be arbitrary. We set $\bg_{K,F}^\ominus :=
    \bg_h^\ominus|_{K,F}$ and $\be_K,\bh_K$ as the corresponding solutions to
  Problem~\ref{pbm:CHDG_local}. 
  By Lemma~\ref{lem:scatter_contract}, we have
  that
  \begin{equation}
    \sum_{F\in\CF_K}\|\bpi^t_{K,F}(\be_K) - \bn_{K,F}\times\bh_K\|_F^2 \le
    \sum_{F\in\CF_K}\|\bg_{K,F}^\ominus\|_F^2,
    \label{eq:ineq_element}
  \end{equation}
  with equality occurring only if $\bg^\ominus_{K,F} = \bzero$ for all $F\in\CF_K$.
  Next, by definition of the scattering operator \eqref{eq:scatter_operator},
  \eqref{eq:ineq_element} is equivalent to
  \begin{displaymath}
    \sum_{F\in\CF_K} \|\opS(\bg_{K,F}^\ominus)\|_F^2 \le \sum_{F\in\CF_K}\|\bg_{K,F}^\ominus\|_F^2,
  \end{displaymath}
  again with equality occurring only if $\bg^\ominus_{K,F} = \bzero$ for all $F\in\CF_K$. Since 
  $\bg_h^\ominus\ne \bzero$, summing this last estimate over all $K\in\CT_h$ yields the result.
\end{proof}
Now, we consider the contraction property for the exchange operator $\opPi$.
\begin{theorem} The exchange operator $\opPi$ is a
  contraction, i.e.
  \begin{displaymath}
    \|\opPi(\bg_h^\ominus)\|_{\partial\CT_h} \le \|\bg_h^\ominus\|_{\partial\CT_h},\quad\forall\bg_h^\ominus\in\BG_h^t.
  \end{displaymath}
  In addition, if $\Gamma\rmsub{I}=\emptyset$, $\opPi$ is an involution and an isometry.
\end{theorem}
\begin{proof}
  These results are straightforward consequences of the definition \eqref{eq:exchange_operator}.
\end{proof}
Putting together the previous two theorems, we arrive at the following.
\begin{corollary}
  \label{cor:strict}
  The operator $\opPi\opS$ is a strict contraction, i.e.
  \begin{displaymath}
    \|\opPi\opS(\bg_h^\ominus)\|_{\partial\CT_h} <
    \|\bg_h^\ominus\|_{\partial\CT_h},\quad\forall\bg_h^\ominus\in\BG_h^t\setminus\{\bzero\}.
  \end{displaymath}
\end{corollary}

\section{Discrete systems and iterative solvers}
\label{sect:disc_local}

The solution procedure consists in applying an iterative method to the hybridized problem
\ref{pbm:trans_weak}, which involves solving the local problem \ref{pbm:CHDG_local} on all elements
for a given incoming transmission variable. In this section, we derive the discrete systems and
describe the corresponding iterative procedures. Throughout this section, for vector-valued
functions, we use the non-bold counterpart of a symbol to denote a single component (e.g.,
$n_{K,F,d}$ for the $d\rmsup{th}$ component of $\bn_{K,F}$), and the notation $[\cdot]_i$ refers to
the $i\rmsup{th}$ entry of an arbitrary vector.

\subsection{Representation of the numerical fields}

The numerical fields are represented with Lagrange basis functions of maximum degree $p$. The
discrete unknowns correspond to the values of the Cartesian components of the fields at finite
element nodes, which simplifies the implementation, see
e.g.~\cite{hesthaven2002nodal,hesthaven2007nodal}. The spatial distribution of the nodes in the
reference tetrahedron is defined using a technique described in \cite{warburton2006explicit}, and
the nodes in the reference triangle are chosen to match the face nodes of the tetrahedron. The
number of nodes per tetrahedron and per triangle is given by $N\rmsub{p}=(p+1)(p+2)(p+3)/6$ and
$N\rmsub{fp}=(p+1)(p+2)/2$, respectively.

Denoting the sets of Lagrange basis functions by $\{\ell_{K,i}(\bx)\}_{i=1\dots N\rmsub{p}}$ on each
element $K$ and $\{\ell_{K,F,i}(\bx)\}_{i=1\dots N\rmsub{fp}}$ on each face $F\in\CF_K$, the
$d\rmsup{th}$ components of the vector-valued polynomials $\be_K, \bh_K, \bg_{K,F}$ are written as
\begin{subequations}
\begin{align}
  e_{K,d}(\bx)
  &= \sum_{i= 1}^{N\rmsub{p}} [\me_{K,d}]_i\:\ell_{K,i}(\bx),
  \quad \text{on each $K$},
  \\
  h_{K,d}(\bx)
  &= \sum_{i= 1}^{N\rmsub{p}} [\mh_{K,d}]_i\:\ell_{K,i}(\bx),
  \quad
  \text{on each $K$},\\
  g^{\ominus}_{K,F,d}(\bx)
  &= \sum_{i= 1}^{N\rmsub{fp}} [\mg^{\ominus}_{K,F,d}]_i\:\ell_{K,F,i}(\bx),
  \quad
  \text{on each $F\in\CF_K$},
  \label{eqn:rep2}
\end{align}
\label{eq:rep}
\end{subequations}
for $d=1,2,3$, where $\me_{K,d}$, $\mh_{K,d}\in\mathbb{C}^{N\rmsub{p}}$ and
$\mg^{\ominus}_{K,F,d}\in\mathbb{C}^{N\rmsub{fp}}$
are the corresponding unknown vectors of DOFs.

Since the transmission variables belong to the tangential trace space \eqref{eq:tantrace_defn} on each face, they can be represented with only two components in local frames associated with the faces.
However, the three Cartesian components are stored to simplify the implementation. Indeed,
storing only the tangential components would correspond to a 33\% reduction in DOFs ($2 \times
N\rmsub{fp}$ instead of $3 \times N\rmsub{fp}$ per face) but would necessitate additional operations and
storage to handle the transformation between the tangent space and the full space.

\subsection{Local systems}

To derive the discrete version of the local problem \ref{pbm:CHDG_local}, we first write it in
Cartesian components. First, for a given a vector-valued function $\bv:\Omega\to\mathbb{C}^3$, we recall the
formulas for the $d\rmsup{th}$ component
of the various operators involved:
\begin{displaymath}
    [\curl\bv]_d
      = \sum_{e,f} \varepsilon_{def} \partial_{x_e} v_f,
      \quad
    [\bn\times\bv]_d
      = \sum_{e,f} \varepsilon_{def} n_e v_f,
    \quad
    [\bpi^t(\bv)]_d
      = v_d - n_d \sum_{e} n_e v_e,
\end{displaymath}
for $d=1,2,3$, where $\varepsilon_{def}$ is the Levi-Civita symbol.
Simple calculations then yield the following problem.
\begin{problem}
\label{eq:localPbmCart}
Find
$([e_{K,1},e_{K,2},e_{K,3}],[h_{K,1},h_{K,2},h_{K,3}]) \in [\CP_p(K)]^3\times[\CP_p(K)]^3$
such that, for all
$([v_{K,1},v_{K,2},v_{K,3}],[w_{K,1},w_{K,2},w_{K,3}]) \in [\CP_p(K)]^3\times[\CP_p(K)]^3$,
\begin{displaymath}
  \left\{
  \begin{aligned}
     & \im\kappa\prodV{e_{K,d}}{v_{K,d}}_K
    - \sum_{e,f} \varepsilon_{def} \prodV{h_{K,d}}{\partial_{x_e} v_{K,f}}_K
    + \sum_{F\in\CF_K}\prodS{\tfrac{1}{2}\big[\bg^{\oplus}_{K,F}\big]_d}{v_{K,d}}_{F}
    \\
     & \qquad = \sum_{F\in\CF_K}\prodS{\tfrac{1}{2}g^{\ominus}_{K,F,d}}{v_{K,d}}_{F},
    \\
     & \im\kappa\prodV{h_{K,d}}{w_{K,d}}_{K}
    + \sum_{e,f} \varepsilon_{def} \prodV{e_{K,d}}{\partial_{x_e} w_{K,f}}_K
    + \sum_{F\in\CF_K}\prodS{\tfrac{1}{2}\big[\bn_{K,F}\times\bg^{\oplus}_{K,F}\big]_d}{w_{K,d}}_{F}
    \\
     & \qquad = - \sum_{F\in\CF_K}\sum_{e,f} \prodS{\tfrac{1}{2}\varepsilon_{def} n_{K,F,e}\,g^{\ominus}_{K,F,f}}{w_{K,d}}_{F},
  \end{aligned}
  \right.
\end{displaymath}
with
\begin{displaymath}
  \begin{aligned}
    \big[ \bg^{\oplus}_{K,F}\big]_d
     & =
    \Big(e_{K,d} - n_{K,F,d} \sum_{e} n_{K,F,e} e_{K,e}\Big)
    - \sum_{e,f} \varepsilon_{def} n_{K,F,e}\, h_{K,f},
    \\
    \big[\bn_{K,F}\times\bg^{\oplus}_{K,F}\big]_d
     & = \sum_{e,f}\varepsilon_{def} n_{K,F,e}\, e_{K,f}
    +  \Big(h_{K,d} - n_{K,F,d}\,\sum_e
    n_{K,F,e}\,h_{K,e}\Big),
  \end{aligned}
\end{displaymath}
for $d=1,2,3$.
\end{problem}

The discrete system is obtained by using \eqref{eq:rep} in the previous problem and by taking the Lagrange functions as test functions.
We introduce the local mass and stiffness matrices
\begin{displaymath}
  \begin{aligned}
    [\MM_{K}]_{ij}
     & = \prodV{\ell_{K,j}}{\ell_{K,i}}_K
     &
     & \text{for } i,j= 1\dots N\rmsub{p},
    \\
    [\MB_{K,F}]_{ij}
     & = \prodS{\ell_{K,F,j}}{\ell_{K,i}}_{F}
     &
     & \text{for } i= 1\dots N\rmsub{p} \text{ and } j= 1\dots N\rmsub{fp},
    \\
    [\MS_{K,d}]_{ij}
     & = \prodV{\partial_{x_d}\ell_{K,j}}{\ell_{K,i}}_K
     &
     & \text{for } i,j= 1\dots N\rmsub{p} \text{ and } d=1,2,3,
  \end{aligned}
\end{displaymath}
and the $N\rmsub{fp}\times N\rmsub{p}$ restriction matrix $\MR_{K,F}$ that maps the nodes associates to an element $K$ to the nodes associated to a face $F\in\CF_K$.
Taking this notation into account, we arrive at the following version of the
local problem.
\begin{problem}[Nodal local system]
Find $([\me_{K,1},\me_{K,2},\me_{K,3}],[\mh_{K,1},\mh_{K,2},\mh_{K,3}]) \in [\mathbb{C}^{N\rmsub{p}}]^3\times[\mathbb{C}^{N\rmsub{p}}]^3$ such that
\begin{displaymath}
  \left\{
  \begin{aligned}
  \im\kappa\MM_K\me_{K,d}
  + \sum_{e,f}\varepsilon_{def}\MS_{K,e}^\top\mh_{K,f}
  + \sum_{F\in\CF_K}\tfrac{1}{2}\MB_{K,F} \mf_{K,d}^\mathrm{E}
  &= \sum_{F\in\CF_K}\tfrac{1}{2}\MB_{K,F} \ms_{K,d}^\mathrm{E},
  \\
  \im\kappa\MM_K\mh_{K,d}
  - \sum_{e,f}\varepsilon_{def}\MS_{K,e}^\top\me_{K,f}
  + \sum_{F\in\CF_K}\tfrac{1}{2}\MB_{K,F} \mf_{K,d}^\mathrm{H}
  &= -\sum_{F\in\CF_K}\tfrac{1}{2}\MB_{K,F} \ms_{K,d}^\mathrm{H},
  \end{aligned}
  \right.
\end{displaymath}
with
\begin{displaymath}
  \begin{aligned}
    \mf_{K,d}^\mathrm{E}
     & := \MR_{K,F}\Big[ \Big(\me_{K,d} - n_{K,F,d} \sum_{e} n_{K,F,e}\me_{K,e}\Big)
    - \sum_{e,f}\varepsilon_{def}n_{K,F,e}\mh_{K,f} \Big],
    \\
    \mf_{K,d}^\mathrm{H}
     & := \MR_{K,F}\Big[ \sum_{e,f}\varepsilon_{def}n_{K,F,e}\me_{K,f}
    + \Big(\mh_{K,d} - n_{K,F,d} \sum_e n_{K,F,e}\mh_{K,e}\Big) \Big],
    \\
    \ms_{K,d}^\mathrm{E}
     & := \Big(\mg_{K,F,d}^{\ominus} - n_{K,F,d}\sum_{e}n_{K,F,e}\mg_{K,F,e}^\ominus\Big),
    \\
    \ms_{K,d}^\mathrm{H}
     & := \sum_{e,f}\varepsilon_{def}n_{K,F,e}\mg^\ominus_{K,F,f},
  \end{aligned}
\end{displaymath}
for $d=1,2,3$, for given data $\{[\mg^\ominus_{K,F,1},\mg^\ominus_{K,F,2},\mg^\ominus_{K,F,3}]\}_{F\in\mathcal{F}_K}$.
\end{problem}

In practice, both equations can be multiplied by the inverse of the local mass matrix, and the local matrices associated to any element $K$ can be easily computed to those associated with the reference element, see e.g.~\cite{hesthaven2002nodal,hesthaven2007nodal}.

\subsection{Hybridized systems}

A discrete version of the hybridized problem \ref{pbm:trans_weak} is directly obtained by using the representation \eqref{eqn:rep2} and by taking the Lagrange basis functions as test functions.
By recasting all the discrete values of $\bg_h^\ominus$ and $\bb_h$ into the global vectors $\mg$ and $\mb$, respectively, we obtain the system
\begin{equation}
  \widetilde{\MA}\mg := \MM(\MI - \MPi\MS)\mg = \MM\mb,
  \label{eq:linsys_lagrange}
\end{equation}
where $\MS$ and $\MPi$ are the algebraic counterparts of the scattering and exchange operators in the nodal basis, $\MI$ is the identity matrix, $\MM$ is a block diagonal matrix with the mass matrices associated to the faces of all the elements.

In our implementation, we have considered two preconditioned versions of System \eqref{eq:linsys_lagrange}.
For the first one, we precondition the system on the left with the mass matrix.
\begin{problem}[Nodal hybridized system]
  Find $\mg\in\mathbb{C}^{N_\mathrm{DOF}}$ satisfying
  \begin{equation}
    \MA\mg := (\MI - \MPi\MS)\mg = \mb.
  \end{equation}
  \end{problem}
For the second version, we have used a symmetric preconditioning with the mass matrix, which corresponds to using modal basis functions that are orthonormal with respect to the inner product $\langle\cdot,\cdot\rangle_{\partial\CT_h}$ instead of nodal basis functions.
Following this interpretation, we introduce the modal versions of the discrete matrices and vectors denoted by $\hat{\MA}$, $\hat{\MPi}$, $\hat{\MS}$, $\hat{\mg}$ and $\hat{\mb}$.
In the modal basis, the system then reads as follows.
\begin{problem}[Modal hybridized system]
Find $\hat{\mg}\in\mathbb{C}^{N_\mathrm{DOF}}$ satisfying
\begin{equation}
  \hat{\MA}\hat{\mg} := (\MI - \hat{\MPi}\hat{\MS})\hat{\mg} =
  \hat{\mb}.
  \label{eq:linsys_orthog}
\end{equation}
\end{problem}
In practice, this system is rewritten with the nodal versions of the matrices and vectors by using a change of basis.
Let introduce the change of basis matrix $\MV$ that maps the modal components of any vector to the nodal ones.
Each column of $\MV$ contains the values of one modal basis function at the finite element nodes, see e.g.~\cite{hesthaven2007nodal}.
Using the relations $\mg = \MV\hat{\mg}$, $\mb = \MV\hat{\mb}$, $\MPi = \MV\hat{\MPi}\MV^{-1}$ and $\MS = \MV\hat{\MS}\MV^{-1}$ in System \eqref{eq:linsys_orthog} gives
\begin{displaymath}
  (\MV^{-1}\MA\MV)(\MV^{-1}\mg) = (\MV^{-1}\mb).
\end{displaymath}
Because the mass matrix verifies $\MM = \MV^{-\top}\MV^{-1}$, we also have
\begin{displaymath}
  (\MV^\top\widetilde{\MA}\MV)(\MV^{-1}\mg) = \MV^\top(\MM\mb),
\end{displaymath}
which shows the link with the preconditioning technique.

In the next section, we shall use the $L^2$-norm of fields of $\BG_h^t$, which can be easily computed in both nodal and modal bases.
For any $\bg_h\in\BG_h^t$, we indeed have
\begin{displaymath}
  \|\bg_h\|_{\partial\CT_h}
  = \sqrt{\prodS{\bg_h}{\bg_h}_{\partial\CT_h}}
  = \|\hat{\mg}\|_2
  = \|\mg\|_{\MM},
\end{displaymath}
with the algebraic norms $\|\hat{\mg}\|_2 := \sqrt{\hat{\mg}\:\!^\top\hat{\mg}}$ and $\|\mg\|_{\MM} := \sqrt{\mg^\top\MM\mg}$.
Because the modal basis is orthonormal with respect to $\langle\cdot,\cdot\rangle_{\partial\CT_h}$, the modal mass matrix is the identity.

\subsection{Iterative linear solvers}
\label{sect:linsolve}

Following \cite{modave2023hybridizable}, we consider iterative solvers based on the fixed-point, CGNR and GMRES iterations.

\paragraph{Fixed-point iteration}
The exchange and scattering matrices inherit from the properties of the corresponding operators, which are proved in \S\ref{sect:analysis:global}.
By Corollary~\ref{cor:strict}, the matrices $\hat{\MPi}\hat{\MS}$ and $\MPi\MS$ are strict contractions in the norms $\left\|\cdot\right\|_2$ and $\left\|\cdot\right\|_{\MM}$, respectively.
In both cases, the spectral radius of the matrix is strictly smaller than $1$, and the system can be solved with the fixed-point iteration.
Writing the resulting algorithm with the nodal matrices and vectors, we obtain, for an initial guess $\mg^{(0)}$,
\begin{displaymath}
  \mg^{(\ell + 1)} = \MPi\MS\bg^{(\ell)}
  + \mb,\quad\ell = 0,1,\dots
\end{displaymath}
This iteration is completely equivalent to the corresponding modal version which can be seen
via the change of basis:
\begin{displaymath}
  \underbrace{\MV^{-1}\mg^{(\ell + 1)}}_{\hat{\mg}^{(\ell + 1)}} =
  \underbrace{\MV^{-1}\MPi\MS\MV}_{\hat{\MPi}\hat{\MS}}
  \underbrace{(\MV^{-1}\bg^{(\ell)})}_{\hat{\mg}^{(\ell)}}
  + \underbrace{\MV^{-1}\mb}_{\hat{\mb}},\quad\ell = 0,1,\dots.
\end{displaymath}

\paragraph{CGNR iterations}

The conjugate gradient normal method (CGNR) is applied to the nodal system $\MA\mg = \mb$, which is equivalent to applying the conjugate gradient algorithm to the normal equation $\MA^*\MA\mg = \MA^*\mb$.
The CGNR iteration is also applied to the modal system, where the normal equation $\hat{\MA}^*\hat{\MA}\hat{\mg} = \hat{\MA}\:\!^*\hat{\mb}$ is rewritten with the nodal matrices and vectors.
These approaches are referred as the \textit{nodal CGNR} and \textit{modal CGNR} solvers, respectively.

These two solvers exhibit qualitatively different numerical properties.
With the nodal solver, the iterate $\mg^{(\ell)}$ minimizes $\mg\mapsto\|\MA\mg - \mb\|_2$ in the affine subspace $\mg^{(0)} + \CK_\ell(\MA^*\MA,\MA^*\mr^{(0)})$ with the initial residual $\mr^{(0)} := \MA\mg^{(0)} - \mb$ for an initial guess $\mg^{(0)}$, see e.g.~\cite{Saad2003}.
With the modal solver, the iterate $\mg^{(\ell)}$ minimizes
$\mg\mapsto\|\MA\mg - \mb\|_{\MM}$
in $\mg^{(0)} + \CK_\ell(\MM^{-1}\MA^*\MM\MA,\MM^{-1}\MA^*\MM\mr^{(0)})$.

The modal CGNR solver is given in Algorithm~\ref{alg:CGNRmodal}.
Compared to the nodal one, it requires one application of the mass matrix and its inverse at each iteration, as well as the use of the $\MM$-norm instead of the 2-norm.
These operations are cheap and easy to parallelize because $\MM$ is block diagonal with small blocks of size $N\rmsub{fp}\times N\rmsub{fp}$.

\begin{algorithm}[!t]
  \caption{Modal CGNR iterative solver}
  \label{alg:CGNRmodal}
  \begin{algorithmic}
    \STATE{\text{Initial guess:} $\mg^{(0)} \in \mathbb{R}^n$}
    \STATE{\text{Initial residual:} $\mr^{(0)} = \mb-\MA\mg^{(0)}$}
    \STATE{$\mz^{(0)} = \MM^{-1}\MA^*\MM\mr^{(0)}$}
    \STATE{$\mmp^{(0)} = \mz^{(0)}$}
    \FOR{$\ell = 0,1,\dots$ until convergence}
    \STATE{$\mq^{(\ell)} = \MA\mmp^{(\ell)}$}
    \STATE{$\alpha^{(\ell)} = \|\mz^{(\ell)}\|_{\MM}^2 / \|\mq^{(\ell)}\|_{\MM}^2$}
    \STATE{$\mg^{(\ell+1)} = \mg^{(\ell)} + \alpha^{(\ell)} \mmp^{(\ell)}$}
    \STATE{$\mr^{(\ell+1)} = \mr^{(\ell)} - \alpha^{(\ell)} \mq^{(\ell)}$}
    \STATE{$\mz^{(\ell+1)} = \MM^{-1}\MA^*\MM\mr^{(\ell+1)}$}
    \STATE{$\beta^{(\ell)} = \|\mz^{(\ell+1)}\|_{\MM}^2 / \|\mz^{(\ell)}\|_{\MM}^2$}
    \STATE{$\mmp^{(\ell+1)} = \mz^{(\ell+1)} + \beta^{(\ell)} \mmp^{(\ell)}$}
    \ENDFOR
  \end{algorithmic}
\end{algorithm}

\paragraph{GMRES iterations}

The generalized minimal residual method (GMRES) is applied to both nodal and modal systems, which we
refer to as the \textit{nodal GMRES} and \textit{modal GMRES} solvers, respectively. The iterate
$\mg^{(\ell)}$ minimizes $\mg\mapsto\|\MA\mg - \mb\|_2$ and $\mg\mapsto\|\MA\mg - \mb\|_{\MM}$,
respectively, in the affine subspace $\mg^{(0)} + \CK_\ell(\MA,\mr^{(0)})$. Again, the cost of a
modal iteration is slightly higher than for a nodal iteration because it requires the application of
the mass matrix and its inverse, as well as the use of the $\MM$ norm instead of the 2-norm.
We will also use the standard practice of restarting the GMRES solver and the notation
  GMRES($n$) means that we restart every $n$ iteration. Finally, we explicitly compute the norm of
  the residual $\|\mb - \MA\mg\|_2$ instead of using the norm of the residual that is provided 
  automatically during the Householder version of a GMRES iteration. Indeed, for the modal version,
  this residual is naturally computed in the $\|\cdot\|_{\MM}$-norm. In our experiments, we did not remark
a major difference between the two so we only present the $\|\cdot\|_2$-norm version.

\section{Numerical results}
\label{sect:numerics}

In this section, we present numerical results to validate our implementation, to study and compare the different iterative solvers, and to illustrate the applicability of the approach on a realistic case.
We have implemented the CHDG method with the different iterative linear solver in a dedicated \texttt{C++} code already used for the acoustic case in \cite{modave2024accelerated}.
This code has parallel capabilities with both \texttt{MPI} and \texttt{OpenMP}.
The meshes generation and visualization are performed with \texttt{gmsh} \cite{Geuzaine2009}. In
what follows, the mesh size parameter $h$ is controlled via
\texttt{MeshSizeMin}/\texttt{MeshSizeMax} in \texttt{gmsh}.

\subsection{Validation and comparison}
\label{sect:comparision}

We consider two reference benchmarks corresponding to a free space case (with an analytic solution) and a cavity case (with a quasi-analytic solution) defined the unit cube $\Omega = (0,1)^3 \subset \mathbb{R}^3$.
For both benchmarks, we use a uniform tetrahedral mesh with $h=0.4$ and polynomial
degree $p=4$. The wavenumber is set to $\kappa = 2.1\pi$.

\paragraph{Benchmark 1 (Free space)}
The reference solution of the first case corresponds to the plane-wave fields given by
$\be\rmsub{ref}(\bx) = \be_0 e^{ \im \kappa_r (\bd\cdot\bx)}$
and
$\bh\rmsub{ref}(\bx) = - \bd \times \be\rmsub{ref}(\bx)$
with the direction vector
$\bd := (1,1,1)^\top/\sqrt{3}$
and the amplitude vector $\be_0 = (0,1,-1)^\top/\sqrt{2}$.
On the boundary we set $\partial \Omega := \Gamma\rmsub{I}$, and impose the
exact solution via $\bs\rmsub{I} := -\bn\times(\bn\times\be\rmsub{ref}) +
  \bn\times\bh\rmsub{ref}$.

\begin{figure}[!t]
  \centering
  \begin{subfigure}[b]{0.40\textwidth}
    \centering
    \includegraphics[width=\linewidth]{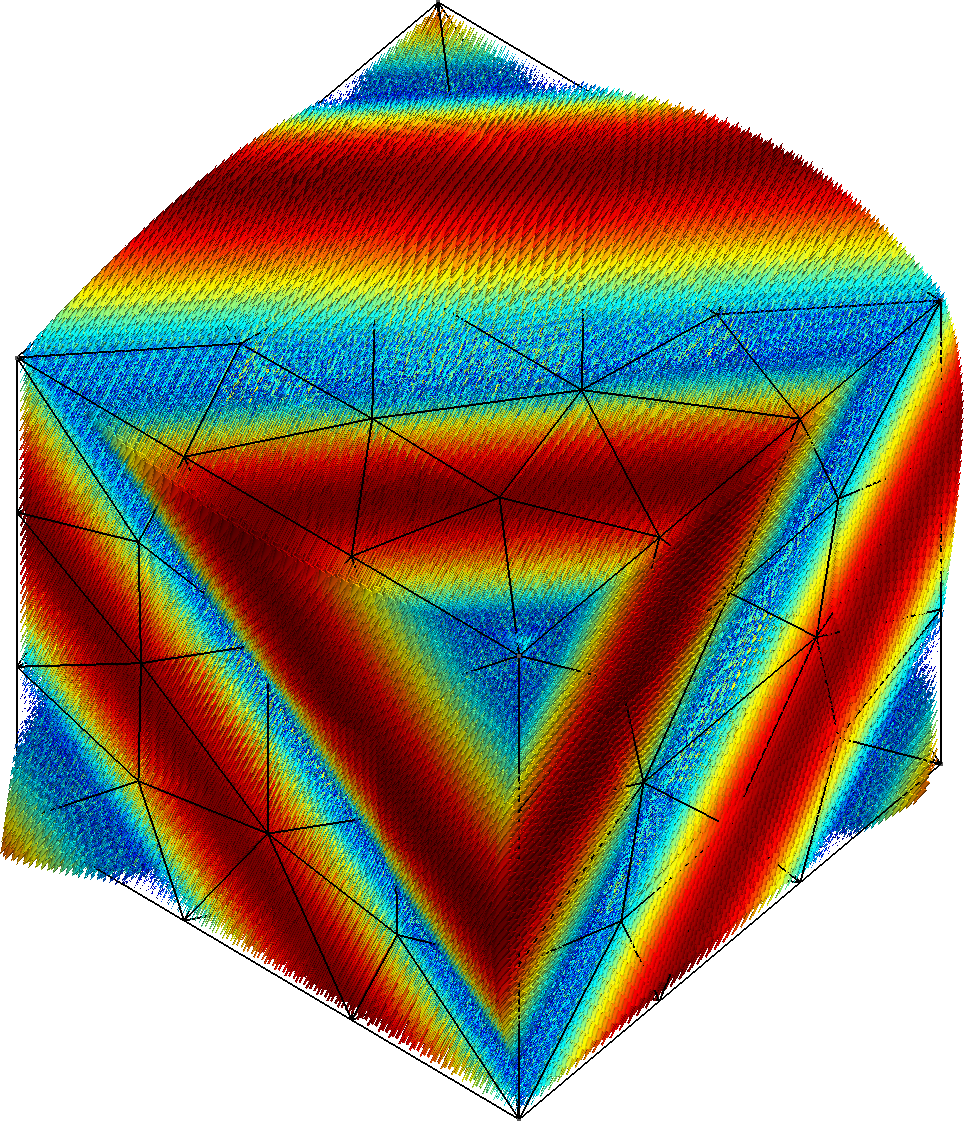}\\[1ex]
    \includegraphics[width=\linewidth]{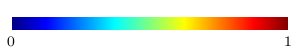}

    \caption{Benchmark~1: free space}
  \end{subfigure}\hfill
  \begin{subfigure}[b]{0.40\textwidth}
    \centering
    \includegraphics[width=\linewidth]{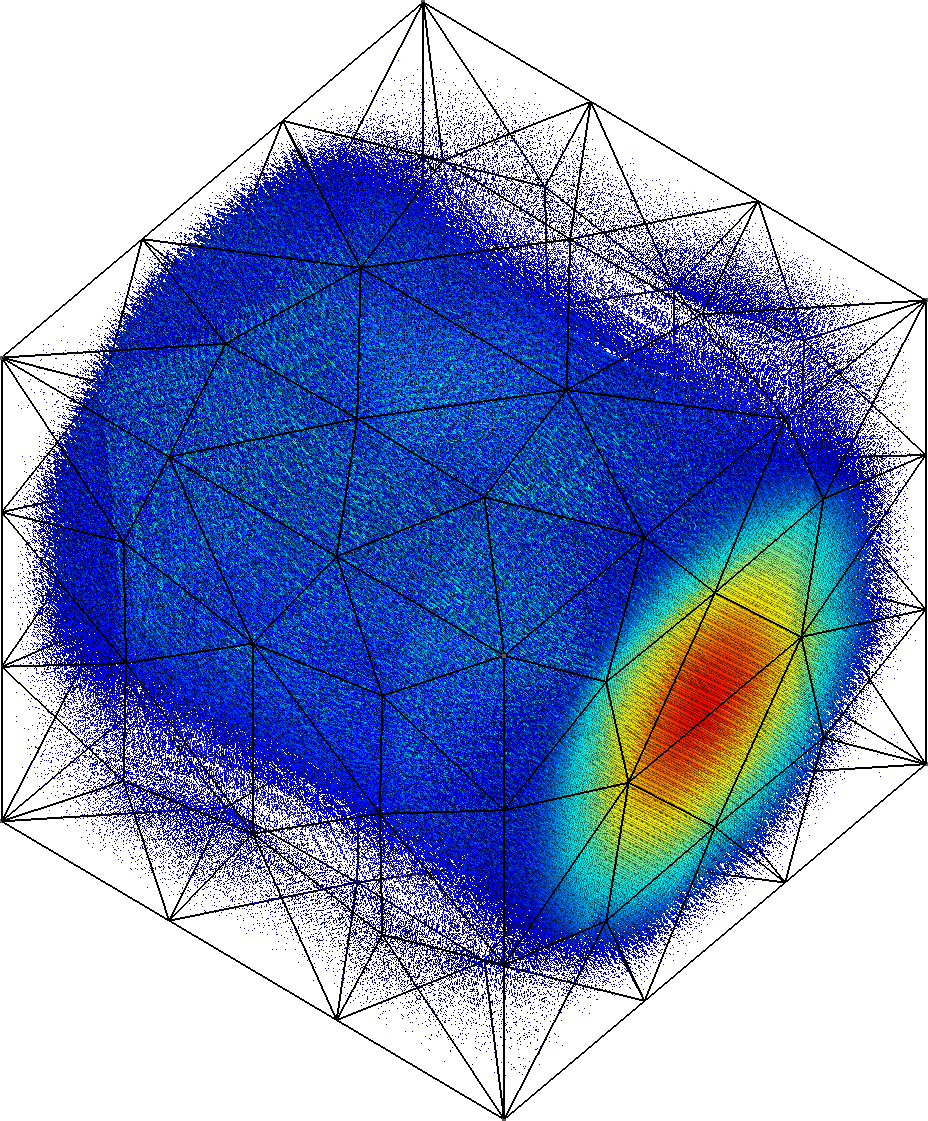}\\[1ex]
    \includegraphics[width=\linewidth]{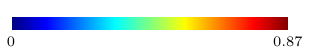}
    \caption{Benchmark~2: cavity}
  \end{subfigure}
  \caption{Plots of $\Re(\be\rmsub{ref})$ for the reference
    solutions. The coloring corresponds to the amplitude of the vector.}
  \label{fig:analytic_profiles}
\end{figure}

\paragraph{Benchmark 2 (Cavity)}
  In the second case, we impose perfect
  electric conductor (PEC) boundary conditions, i.e.  $\partial\Omega :=
    \Gamma\rmsub{E}$ with $\bs\rmsub{E} = \bzero$, and the system is driven by an
  external (volumetric) electric current density $\BJ\rmsub{e} = (-\im/\kappa,
    0, 0)^\top$.
  This problem admits a semi-analytic solution by considering an expansion
  of the eigenmode solutions for the electric field, i.e.
  \[
    \be\rmsub{ref}(\bx) = \sum_{k_1,k_2,k_3\in\mathbb{N}} \alpha_{k_1,k_2,k_3}
    \begin{pmatrix}
      a_1 \cos(k_1 \pi x) \sin(k_2 \pi  y) \sin(k_3 \pi z) \\
      a_2 \sin(k_1 \pi x) \cos(k_2 \pi  y) \sin(k_3 \pi z) \\
      a_3 \sin(k_1 \pi x) \sin(k_2 \pi  y) \cos(k_3 \pi z)
    \end{pmatrix}
  \]
  with \((a_1,a_2,a_3)\cdot(k_1,k_2,k_3)=0\) in order to ensure \(\nabla\cdot \be\rmsub{ref}=0\).
  For the chosen current density $\BJ\rmsub{e}$, the eigenmodes simplify substantially for $k_1 \equiv 0$ and $k_2, k_3$ odd, and the coefficients of the expansion can be calculated explicitly, leading to 
  \[
    \be\rmsub{ref}(\bx) = \sum_{\substack{k_2,k_3\in\mathbb{N}\:\text{odd}}}
    \frac{16}{\pi^2 k_2 k_3(\pi^2 (k_2^2 + k_3^2) - \kappa_r^2)}
    \begin{pmatrix}
      \sin(k_2 y) \sin(k_3 z) \\ 0 \\ 0
    \end{pmatrix}.
  \]
  In practice, we retain up to $k_2,k_3 = 25$ in the series to represent the reference solution.
  \label{bench:cavity}

\paragraph{Relative error and residual}
At each iteration $\ell$, we consider the error of the physical fields for
$\be_h\!^{(\ell)},\bh_h\!^{(\ell)}\in\BV_h$ corresponding to the global vector $\mg^{(\ell)}$
\begin{align*}
  \text{Relative error } := \sqrt{
    \frac
    {\|\be_h\!^{(\ell)} - \be\rmsub{ref}\|^2_{L^2(\Omega)} + \|\bh_h\!^{(\ell)} -  \bh\rmsub{ref}\|^2_{L^2(\Omega)}}
    {\|\be\rmsub{ref}\|^2_{L^2(\Omega)} + \|\bh\rmsub{ref}\|^2_{L^2(\Omega)}}}
\end{align*}
and the relative residual in 2-norm
\begin{align*}
  \text{Relative residual } := 
    \frac{\|\mb - \MA\mg^{(\ell)}\|_2}{\|\mb - \MA\mg^{(0)}\|_2}.
\end{align*}
These quantities are plotted as a function of iterations in Figure~\ref{fig:comparision} for both benchmarks.

\begin{figure}[!tb]
  \centering
  \includegraphics[width=\linewidth]{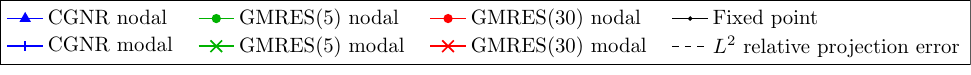}
  \\[0.3cm]
  \begin{subfigure}[b]{0.47\textwidth}
    \centering
    \includegraphics[width=\linewidth]{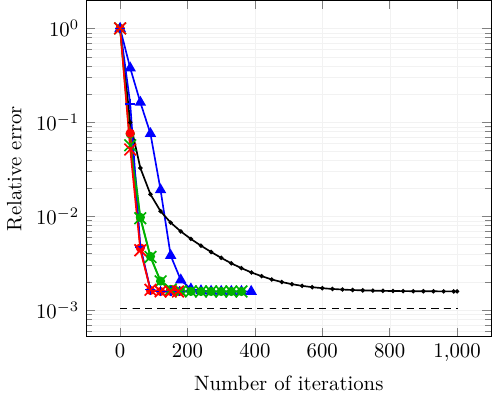}\\[1ex]
    \caption{Relative error history for Benchmark~1}
  \end{subfigure}\hfill
  \begin{subfigure}[b]{0.47\textwidth}
    \centering
    \includegraphics[width=\linewidth]{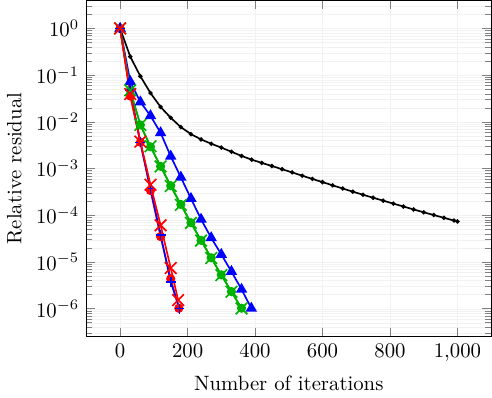}\\[1ex]
    \caption{Relative residual history for Benchmark~1}
  \end{subfigure}
  \\[0.3cm]
  \begin{subfigure}[b]{0.47\textwidth}
    \centering
    \includegraphics[width=\linewidth]{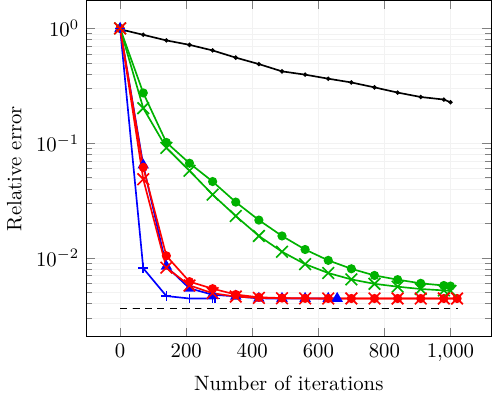}\\[1ex]
    \caption{Relative error history for Benchmark~2}
  \end{subfigure}\hfill
  \begin{subfigure}[b]{0.47\textwidth}
    \centering
    \includegraphics[width=\linewidth]{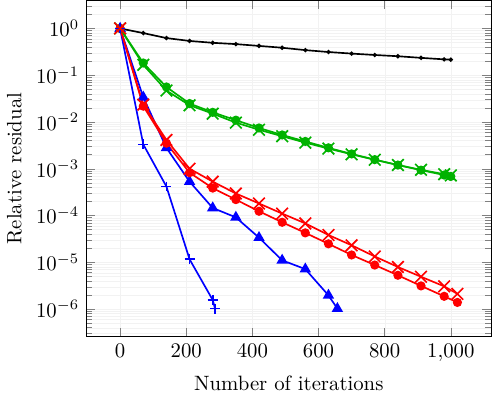}\\[1ex]
    \caption{Relative residual history for Benchmark~2}
  \end{subfigure}
  \caption{Results of the different iterative solvers for the
  Benchmarks~1 (free space) and 2 (cavity).}
  \label{fig:comparision}
\end{figure}

The $L^2$ relative projection error, corresponding to the relative error between $\be\rmsub{ref},\bh\rmsub{ref}$ and their $L^2$ projections on the finite element space, is represented by the horizontal dashed lines.
We observe that none of the methods reduce the error to the level of the projection error.
This is expected since the numerical solution we produce is rigorously equivalent to the DG solution which in turn is not equal to the $L^2$ projection of the true solution.
However, the final error produced the CHDG discretization is about twice of the relative
projection error.

\paragraph{Comparison of the iterative solvers}
We now compare the performance of different iterative solvers.
Several remarks are in order with respect to the results:
\begin{itemize}
\item The fixed-point iteration converges for both test cases as guaranteed by
        Corollary~\ref{cor:strict}. However, we see that this convergence is much
        slower for the cavity benchmark, which can be partially explained by the fact
        that the proximity of resonance modes results in
        worse conditioning at the discrete level.
\item Both versions of the CGNR method outperform the fixed-point iteration on
        both benchmarks. This is more pronounced for the cavity benchmark due to the
        previous remark. Moreover, the modal CGNR as in \eqref{alg:CGNRmodal} is
        more efficient for both benchmarks. As discussed in \S\ref{sect:linsolve},
        we attribute this to the fact the CGNR is minimizing the residual in the
        same norm that we have the contraction.
      \item The sequence of iterates produced by GMRES(30) converges more quickly than
        both the fixed-point and nodal GGNR variant. In the case of the free space benchmark, the
        behavior of GMRES(30) and the modal CGNR variant is almost identical. On the other
        hand, for the cavity benchmark, the modal CGNR variant is more efficient. We also
          consider a shorter restart interval given by GMRES(5) and observe a clear deterioration in
          the speed of convergence for both benchmarks. We tested an non-restarted version of GMRES
          as well which performed practically identically to GMRES(30) for the plane wave benchmark,
          and was faster than CGNR nodal but slower than CGNR modal for the cavity benchmark.
\end{itemize}
These results are consistent with those obtained in the acoustic case in \cite{modave2023hybridizable}.

\subsection{Solver results for the realistic test case}
\label{sect:cobra}

We now present results with a benchmark designed and measured by EADS Aerospatiale Matra Missiles for Workshop EM-JINA 98.
Numerical tests have been performed for the same setup in, e.g. \cite{Liu2003, Dolean2015, Vion2018, Bootland2021}.

\begin{figure}
  \centering
  \includegraphics[width=0.45\textwidth]{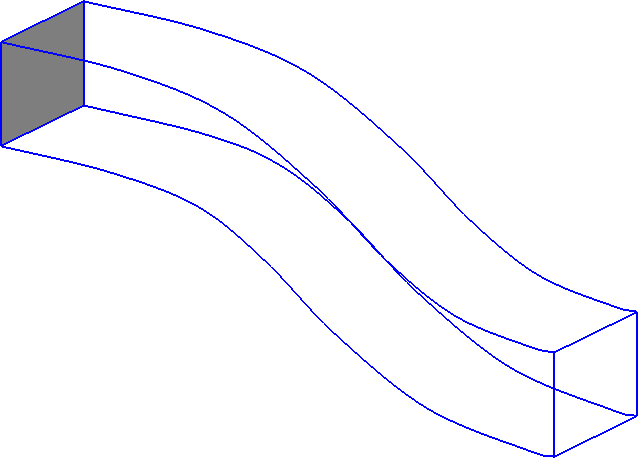}
  \qquad
  \includegraphics[width=0.45\textwidth]{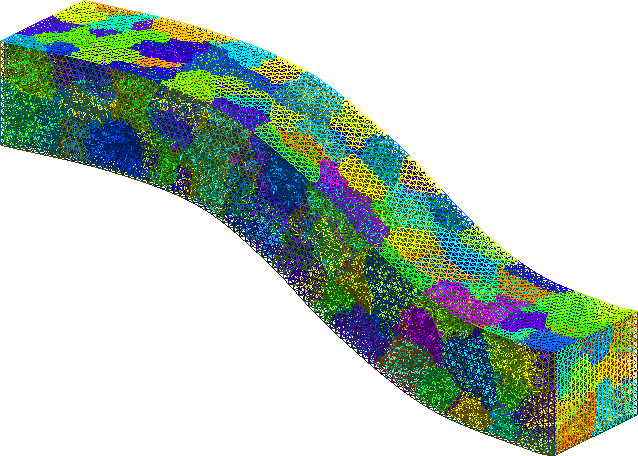}
  \caption{Geometry of the COBRA domain with the aperture in gray (left) and the
    subdomain partition with 160 subdomains (right).}
  \label{fig:cobra}
\end{figure}

\begin{figure}[!h]
  \centering
  \begin{subfigure}[b]{0.45\textwidth}
    \centering
    \includegraphics[width=\linewidth]{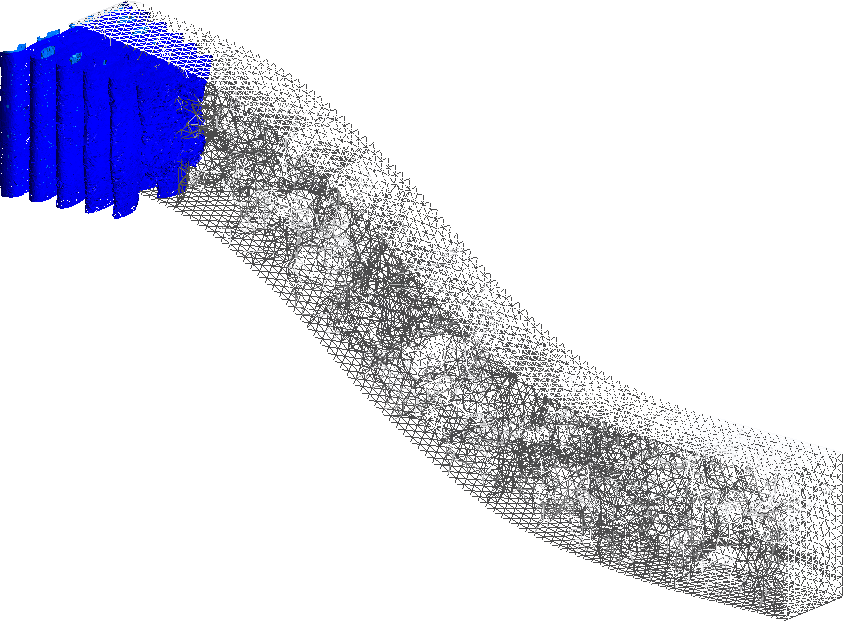}
    \caption{100 iterations}
  \end{subfigure}%
  \begin{subfigure}[b]{0.45\textwidth}
    \centering
    \includegraphics[width=\linewidth]{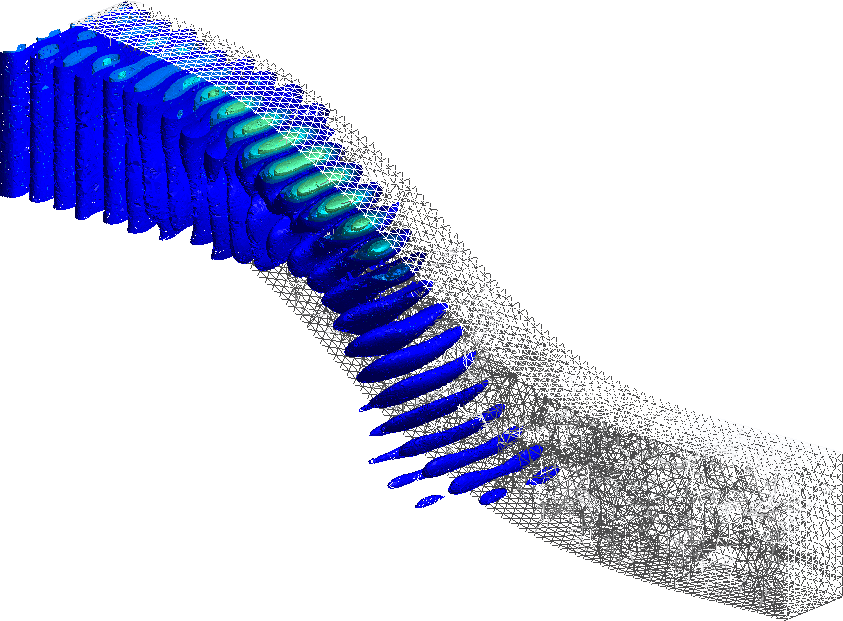}
    \caption{400 iterations}
  \end{subfigure}
  \begin{subfigure}[b]{0.45\textwidth}
    \centering
    \includegraphics[width=\linewidth]{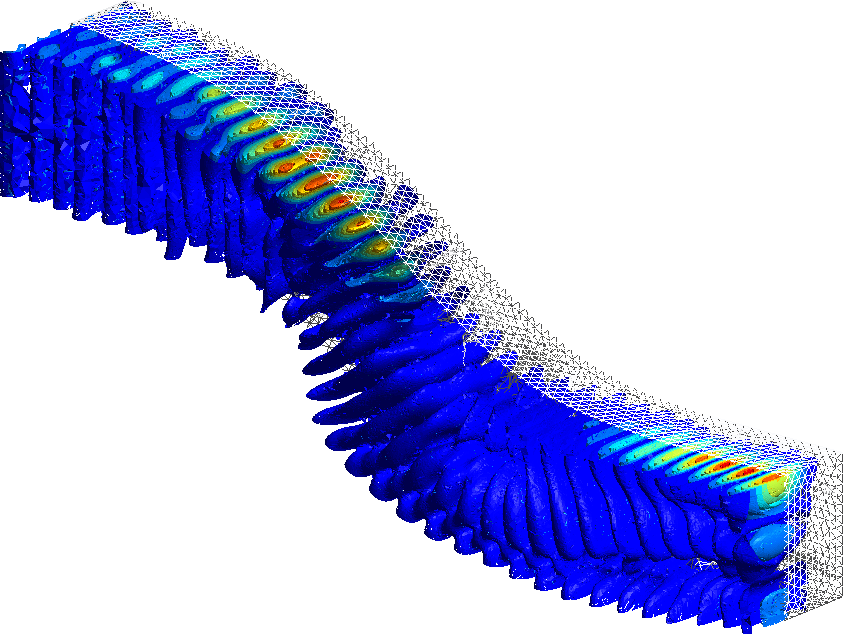}
    \caption{1000 iterations}
  \end{subfigure}
  \begin{subfigure}[b]{0.45\textwidth}
    \centering
    \includegraphics[width=\linewidth]{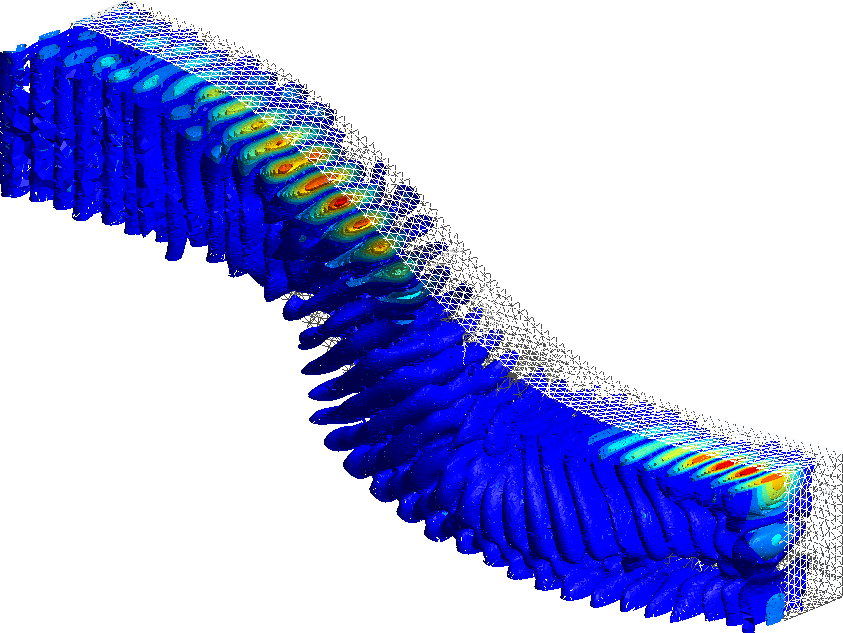}
    \caption{1600 iterations}
  \end{subfigure}
  \includegraphics{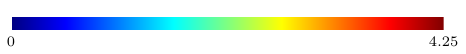}
  \caption{Snapshots of $\left|\Re(\be_h)\right|$ for the COBRA benchmark
    at four different points during the solution process using modal CGNR, cf.
    \S\ref{sect:cobra}.}
  \label{fig:cobra_solution}
\end{figure}

The domain $\Omega$ consists of five continuous segments.
The first segment is a straight rectangular section with a length of \SI{10}{mm}, followed by a circular bend of \SI{35}{\degree} with a radius of \SI{186}{mm}.
The third segment is another straight section measuring \SI{80}{mm} in length, which is then followed by a second circular bend, also \SI{35}{\degree} with a radius of \SI{186}{mm}.
Finally, the domain ends with a straight segment of \SI{100}{mm}.
The domain has a uniform cross-section measuring \SI{84}{mm} by \SI{110}{mm}.
We impose PEC boundary conditions $\bs\rmsub{E}=\bzero$ on all the boundaries except the aperture as shown in Figure~\ref{fig:cobra}.
At the aperture we prescribe an incident free space with a frequency of \SI{10}{GHz} corresponding to $\kappa_r = \SI{209}{m^{-1}}$.
We use a uniform mesh with $h=6\times 10^{-3}$ and polynomial degree $p= 3$ which respects approximately 5 cells per wavelength, resulting in
$13,308,240$ DOFs for the nodal eliminated system \eqref{eq:linsys_lagrange}.

In order to render this computation tractable, we use a \texttt{METIS}
\cite{Karypis_1998} mesh partition with 160 subdomains as depicted in
Figure~\ref{fig:cobra}. This partition was then used on a local cluster with
four nodes, each equipped with 40 Intel Xeon CPU Gold 6230 cores at 2.1 GHz for a
total of 160 \texttt{MPI} ranks.

In Figure~\ref{fig:cobra_solution} we plot the magnitude of the
real part of the electric field for the solution obtained at intermediate solver iterations
for the modal CGNR method. After 2000 iterations the solution stabilizes and we
obtain a profile similar to the one obtained e.g. in \cite{Dolean2015}, with
higher magnitude near the first bent region after the aperture and near the
``closed'' (PEC) side of the cavity.

\begin{figure}[!tb]
  \centering
  \includegraphics[width=0.6\linewidth]{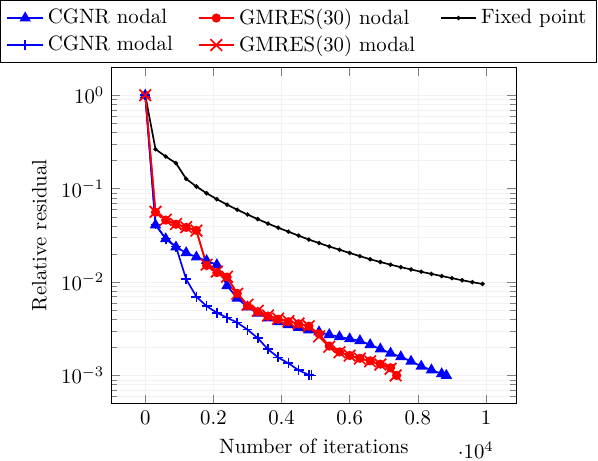}
  \caption{Relative residual history for the COBRA benchmark.}
  \label{fig:cobra_residual}
\end{figure}

The residual history of the different methods is plotted in Figure~\ref{fig:cobra_residual}. We
remark a much slower convergence as compared to the previous test cases, which is reasonable since
the linear system size is also much greater and the configuration is more complex. We impose a tolerance of $10^{-3}$ on the relative
residual for the stopping criteria. The fixed point iteration failed to converge after  15 hours of
wall time corresponding to $4.5\times 10^4$ iterations and a final relative residual of $2.1 \times 10^{-3}$.

  We give more details of the solvers performance in Table~\ref{tab:runtimes}. In terms of iteration
  count, CGNR modal was the fastest to converge, taking approximately 66\% of the iterations of the next faster
  method (GMRES(30) nodal). However, we see that the average time per iteration is higher for both the CGNR variants,
  practically double that of GMRES(30) and fixed point. This is reasonable since CGNR requires the
  action of $\MA$ and $\MA^*$ at each iteration whereas the other methods only require the action of
  $\MA$. There is also a cost associated with the modal versions compared with their nodal
  counterparts, which turned out to be around 4\% of the overall cost of an iteration for both CGNR and
  GMRES(30). We emphasize that these results are preliminary and that further optimizations to the
  code are underway.

\begin{table}[h]
  \centering
  \scalebox{0.95}{
  \begin{tabular}{lccc}
    \hline
    \textbf{Method} & \textbf{\# iterations} & \textbf{Total runtime} & \textbf{Average time/iteration}\\
    \hline
    GMRES(30) nodal & 7379  & 2.74 hours  & \SI{1.34}{s} \\
    GMRES(30) modal & 7818  & 3.02 hours  & \SI{1.39}{s} \\
    CGNR nodal      & 8838  & 6.26 hours  & \SI{2.55}{s} \\
    CGNR modal      & 4876  & 3.59 hours  & \SI{2.65}{s} \\
    Fixed point & X & X & \SI{1.29}{s} \\
    \hline
  \end{tabular}
}
  \caption{Runtime analysis for the COBRA benchmark of \S\ref{sect:cobra} to achieve a relative
  residual of $10^{-3}$. The fixed point iteration failed to achieve this criteria.}
  \label{tab:runtimes}
\end{table}

\section{Conclusion}
\label{sect:conclusion}

In this paper, we developed and analyzed an extension of the CHDG method, originally studied in \cite{modave2023hybridizable} for the Helmholtz equation, to the time-harmonic Maxwell equations.
The main mathematical novelty is the introduction of a vector-valued transmission variable that lives in the tangential trace space of element faces.
The same types of mathematical arguments lead to a hybridized equation of the form $(\opI - \opPi\opS)\bg = \bb$ (Problem~\ref{pbm:trans_operator}).
We prove using similar techniques that the operator $\opPi\opS$ is a strict contraction, and that the local sub-problems arising from the hybridization are well-posed.
A fixed-point iteration can be used to solve the hybridized problem, which is illustrated in our numerical experiments.

Practically, we consider 3D vector cases, while in \cite{modave2023hybridizable} only 2D scalar
cases were presented.
In order to solve the hybridized system efficiently, careful considerations have to be made from an implementation point of view.
In this paper we present a nodal implementation, using ideas borrowed from \cite{hesthaven2007nodal} for discretization and mapping between nodal and modal polynomial bases.
We validate and study variants of our implementation by solving two reference benchmarks, and we also present a large-scale benchmark with a parallel \texttt{C++} code.

We compare iterative strategies with CGNR and (restarted) GMRES iterations based on both nodal and
modal bases, but implemented only with the nodal basis for computational performance as well as ease
of implementation. The modal and nodal versions of the GMRES iterations yield similar results, which
are similar or slightly better than those obtained with the modal CGNR iteration. However, when
using the restarted version of GMRES, modal CGNR slightly outperforms GMRES. We believe that CGNR
can be a competitive approach for practical cases, notably due to its low memory footprint.

In light of the recent work \cite{pescuma2025} where an extension to
heterogeneous media for Helmholtz equation is presented, an obvious extension
would be to consider heterogeneous media in the context of Maxwell's equations.
We are also working to implement a version of the code that is accelerated by
GPUs which could provide further performance gains.

\bibliographystyle{siamplain}
\bibliography{main}

\end{document}

%% file: headerSIAM.tex
\usepackage[a4paper,margin=1.25in]{geometry}

\usepackage{amsfonts}
\usepackage{bm}
\usepackage{graphicx}
\usepackage{epstopdf}
\usepackage{algorithmic}
\ifpdf
  \DeclareGraphicsExtensions{.eps,.pdf,.png,.jpg}
\else
  \DeclareGraphicsExtensions{.eps}
\fi
\usepackage{caption}
\usepackage{subcaption}
\usepackage{siunitx}
\usepackage{pgfplots}
\pgfplotsset{every axis/.append style={
      label style={font=\scriptsize},
      tick label style={font=\scriptsize},
      legend style={font=\tiny\sf, row sep=-2pt},
      legend cell align={left}
    }}
\pgfplotsset{compat=1.18}

\newsiamthm{problem}{Problem}
\newsiamremark{benchmark}{Benchmark}

\headers{CHDG for the time-harmonic Maxwell equations}{A. E. Rappaport, T. Chaumont-Frelet, A. Modave}

\title{A hybridizable discontinuous Galerkin method with transmission variables for time-harmonic electromagnetic problems\thanks{Distributed under \href{https://creativecommons.org/licenses/by/4.0/}{Creative Commons CC BY 4.0} license.
\funding{This work was funded in part by the \textit{ANR JCJC project WavesDG} (research grant ANR-21CE46-0010) and by the \textit{Agence de l'Innovation de Défense} [AID] through \textit{Centre Interdisciplinaire
d'Etudes pour la D\'efense et la S\'ecurit\'e} [CIEDS] (project 2022 ElectroMath). The authors thank the IDCS research facility of \textit{Ecole Polytechnique} for providing access to the \texttt{Cholesky} cluster.}}}

\author{A. E. Rappaport\thanks{POEMS, CNRS, Inria, ENSTA, Institut Polytechnique de Paris, 91120 Palaiseau, France 
  (\email{ari.rappaport@ensta.fr}, \email{axel.modave@ensta.fr}).}
\and T. Chaumont-Frelet\thanks{Inria Univ.~Lille and Laboratoire Paul Painlev\'e, 59655 Villeneuve-d'Ascq, France 
  (\email{theophile.chaumont@inria.fr}).}
\and A. Modave\footnotemark[2]}

%% file: headerPerso.tex
\newcommand{\matalg}[1]{\boldsymbol{\mathsf{#1}}}

\newcommand{\MA}{\matalg{A}}
\newcommand{\MB}{\matalg{B}}

\newcommand{\MI}{\matalg{I}}

\newcommand{\MM}{\matalg{M}}

\newcommand{\MR}{\matalg{R}}
\newcommand{\MS}{\matalg{S}}

\newcommand{\MV}{\matalg{V}}

\newcommand{\MPi}{\matalg{\Pi}}

\newcommand{\mb}{\matalg{b}}

\newcommand{\me}{\matalg{e}}
\newcommand{\mf}{\matalg{f}}
\newcommand{\mg}{\matalg{g}}
\newcommand{\mh}{\matalg{h}}
\newcommand{\mmp}{\matalg{p}}
\newcommand{\mq}{\matalg{q}}
\newcommand{\mr}{\matalg{r}}
\newcommand{\ms}{\matalg{s}}

\newcommand{\mz}{\matalg{z}}

\newcommand{\prodV}[2]{\big({\textstyle{#1}},{\textstyle{#2}}\big)}
\newcommand{\prodS}[2]{\big\langle{\textstyle{#1}},{\textstyle{#2}}\big\rangle}

\renewcommand{\Re}{\operatorname{Re}}

\newcommand{\im}{\imath}

\newcommand{\grad}{\boldsymbol \nabla}

\newcommand{\curl}{\grad \times}

\newcommand{\BG}{\boldsymbol{G}}

\newcommand{\BJ}{\boldsymbol{J}}

\newcommand{\BL}{\boldsymbol{L}}

\newcommand{\BV}{\boldsymbol{V}}

\newcommand{\opS}{\mathbf{S}}
\newcommand{\opI}{\mathbf{I}}
\newcommand{\opPi}{\boldsymbol{\Pi}}

\newcommand{\bb}{\boldsymbol{b}}

\newcommand{\bd}{\boldsymbol{d}}
\newcommand{\be}{\boldsymbol{e}}
\renewcommand{\bf}{\boldsymbol{f}}
\newcommand{\bg}{\boldsymbol{g}}
\newcommand{\bh}{\boldsymbol{h}}

\newcommand{\bn}{\boldsymbol{n}}

\newcommand{\bq}{\boldsymbol{q}}

\newcommand{\bs}{\boldsymbol{s}}

\newcommand{\bv}{\boldsymbol{v}}
\newcommand{\bw}{\boldsymbol{w}}
\newcommand{\bx}{\boldsymbol{x}}

\newcommand{\bzero}{\mathbf{0}}

\newcommand{\CF}{\mathcal{F}}

\newcommand{\CK}{\mathcal{K}}

\newcommand{\CP}{\mathcal{P}}

\newcommand{\CT}{\mathcal{T}}

\newcommand{\BCP}{\boldsymbol{\CP}}

\newcommand{\bpi}{\boldsymbol{\pi}}
\newcommand{\bxi}{\boldsymbol{\xi}}

\newcommand{\rmsub}[1]{_{\mathrm{#1}}}
\newcommand{\rmsup}[1]{^{\mathrm{#1}}}

\newcommand{\ntimes}[1]{\bn\times #1}
\newcommand{\ntimesKF}[1]{\bn_{K,F}\times #1}